%% file: amsdraft.tex
\documentclass{amsart}

\usepackage{url}
\usepackage[breaklinks]{hyperref}

\usepackage{scalerel,stackengine}
\stackMath

\usepackage[T1]{fontenc}
\usepackage[utf8]{inputenc}
\usepackage{enumerate}
\usepackage{romannum}
\usepackage{multirow}
\usepackage{graphicx}
\usepackage{amsmath, amsthm, amssymb, amsfonts}
\usepackage{xparse}
\usepackage{relsize}
\usepackage{tikz-cd}
\usepackage{pgfplots}
\usepackage{subfig}
\usepackage[title,titletoc,toc,page]{appendix}
\usepackage{graphics}
\usepackage{dsfont}
\usepackage{makecell}
\usepackage{cancel}
\usepackage{algorithmic}
\usepackage{algorithm}
\newcommand{\C}{{\mathbb C}}
\newcommand{\F}{{\mathbb F}}

\newcommand{\Q}{{\mathbb Q}}
\newcommand{\R}{{\mathbb R}}
\newcommand{\Z}{{\mathbb Z}}

\newcommand{\calA}{{\mathcal A}}

\newcommand{\calH}{{\mathcal H}}

\newcommand{\calM}{{\mathcal M}}

\newcommand{\calO}{{\mathcal O}}
\newcommand{\cO}{{\mathcal O}}

\newcommand{\calR}{{\mathcal R}}

\newcommand{\calU}{{\mathcal U}}
\newcommand{\calV}{{\mathcal V}}
\newcommand{\calW}{{\mathcal W}}

\newcommand{\fraka}{{\mathfrak a}}
\newcommand{\frakb}{{\mathfrak b}}
\newcommand{\frakc}{{\mathfrak c}}

\newcommand{\frakm}{{\mathfrak m}}

\newcommand{\frakp}{{\mathfrak p}}

\newcommand{\frakD}{{\mathfrak D}}

\newcommand{\frakH}{{\mathfrak H(K^r,\Phi^r)}}

\newcommand{\frakP}{{\mathfrak P}}

\DeclareMathOperator{\Tr}{Tr}

\DeclareMathOperator{\im}{Im}

\DeclareMathOperator{\End}{End}

\DeclareMathOperator{\Aut}{Aut}
\DeclareMathOperator{\Gal}{Gal}

\DeclareMathOperator{\Jac}{Jac}

\DeclareMathOperator{\Princ}{Princ}

\DeclareMathOperator{\spn}{span}

\newcommand{\into}{\hookrightarrow}

\newcommand\wwidehat[1]{\savestack{\tmpbox}{\stretchto{\scaleto{\scalerel*[\widthof{\ensuremath{#1}}]{\kern-.6pt\bigwedge\kern-.6pt}{\rule[-\textheight/2]{1ex}{\textheight}}}{\textheight}}{0.5ex}}\stackon[1pt]{#1}{\tmpbox}}

\newtheorem{theorem}{Theorem}[section]
\newtheorem{lemma}{Lemma}[section]
\newtheorem{proposition}{Proposition}[section]

\newtheorem{corollary}{Corollary}[section]

\theoremstyle{definition}
\newtheorem{definition}[theorem]{Definition}

\theoremstyle{remark}
\newtheorem{remark}[theorem]{Remark}
\newtheorem{notation}[theorem]{Notation}

\numberwithin{equation}{section}

\begin{document}

\title{Genus 3 hyperelliptic curves with CM via Shimura reciprocity}


\author{Bogdan Adrian Dina}
\address{Ulm University and Université de Picardie Jules Verne}
\curraddr{Institute of Theoretical Computer Science, Ulm, Germany}
\email{bogdan.dina@uni-ulm.de}
\thanks{}

\author{Sorina Ionica}
\address{Université de Picardie Jules Verne}
\curraddr{33 Rue Saint Leu Amiens 80039, France}
\email{sorina.ionica@u-picardie.fr}
\thanks{}

\subjclass[2010]{Primary }


\date{}

\dedicatory{}

\begin{abstract}
  Up to isomorphism over $\C$, every simple principally polarized abelian variety of dimension 3 is the Jacobian of a smooth projective curve of genus 3. Furthermore, this curve is either a hyperelliptic curve or a plane quartic. Given a sextic CM field $K$, we show that if there exists a hyperelliptic Jacobian with CM by $K$, then all principally polarized abelian varieties that are Galois conjugated to it are hyperelliptic. Using Shimura's reciprocity law, we give an algorithm for computing approximations of the invariants of the initial curve, as well as their Galois conjugates. This allows us to define and compute class polynomials for genus 3 hyperelliptic curves with CM.
  
\end{abstract}

\maketitle

\input{introduction_ants}

\input{background_ants}

\input{algorithm_ants}

\input{rosenhains_ants}

\input{classpol_ants}

\input{implementation_ants}

\bibliographystyle{plain}
\bibliography{biblio_ants}
\input{appendix_ants}

\newpage
\input{supplementary_material}

\end{document}

%% file: introduction_ants.tex
\section{Introduction}

Shimura and Taniyama's complex multiplication theory shows that it is possible to construct certain abelian extensions of CM fields by computing the values of Siegel modular functions evaluated at points with CM in the Siegel upper half-space. In addition, the effective computation of these modular forms makes it possible to compute models for CM curves, and also to effectively  construct the related class fields.
 

For example, in genus one, the field of modular functions of level 1 is generated by the $j$-invariant. It is well known that the $j$-invariant of an elliptic curve with endomorphism ring $\mathcal{O}_K$ generates the Hilbert 
class field of $K$. In the genus 2 case, the field of Siegel modular functions of level 1 is generated by the absolute Igusa invariants \cite{igusa1962}. Similarly, when evaluated at CM points, their values give invariants of hyperelliptic curves whose Jacobian has CM, and the class field equations, known as class polynomials, are recovered by computing these invariants for all curves with CM by the field~\cite{Streng2014,Thome}. In genus 3, up to isomorphism over $\C$, every simple principally polarized abelian variety (p.p.a.v.) of dimension 3 is the Jacobian of a complete smooth projective curve. Since two different sets of invariants for both genus 3 hyperelliptic curves and plane quartics are known in the literature, it is was not understood until now how to compute class polynomials for genus 3.  


In \cite[Lemma 4.5]{Weng}, Weng shows that a simple principally polarized abelian threefold with CM by a sextic CM field containing $\Q(i)$ is a hyperelliptic Jacobian. In the same paper, Weng gives an algorithm to compute hyperelliptic 
curves whose Jacobian has CM by a sextic field containing $\Q(i)$. In later
work, Balakrishnan \textit{et al.}~\cite{BILV} give an algorithm which removes this restriction on the CM field, by performing an heuristic check. This heuristic relies on Mumford's vanishing criterion~\cite{Mumford2, poor}, which states that a genus 3 curve is hyperelliptic if and only if one of the 36 even theta constants is 0. Given a period matrix with CM by a sextic CM field, the algorithm in~\cite{BILV} first computes the theta constants with enough precision to see if there is one which approximates zero, and then computes the Rosenhain invariants. These invariants generate a certain subfield of the ray class field of modulus 2 of the reflex field $K^r$ of $K$ and by approximating them with high precision, we can recognize them as algebraic numbers. This method has its limitations, since as soon the degree of the class field over which the Rosenhains are defined is high, the complexity of the algebraic dependance computation becomes a bottleneck. From a concrete point of view, only examples of CM fields with class number 1 were considered in~\cite{BILV}.

In this paper, we extend the work of Balakhrishnan \textit{et al.} by considering the action of the Galois group $Gal(CM_{\frakm}(K^r)/K^r)$, with $CM_{\frakm}(K^r)$ a subfield of the ray class field  of a given modulus $\mathfrak{m}$, on a hyperelliptic CM point and showing that the p.p.a.v. obtained in this way are also isomorphic to hyperelliptic Jacobians. Consequently, once we identify a hyperelliptic curve by verifying computationally the Vanishing Criterion condition, we compute the Galois conjugates of its Rosenhain invariants via Shimura's reciprocity law.
With this in hand, we were able to define class polynomials for the Shioda and Rosenhain invariants of genus 3 hyperelliptic curves with CM as being the polynomials whose roots are the Shioda, and Rosenhain invariants respectively, of all conjugate CM points.


Aiming to implement our results in Sage~\cite{sagemath} and compute examples for the class polynomials of the Rosenhain and Shioda invariants, we also propose effective methods to construct the reflex field associated to a given CM type, the typenorm, as well as the image of the typenorm as a subgroup in the Shimura class group.

This paper is organized as follows. Section~\ref{background} sets the notations on the theta constants and modular forms that we use and recalls the Vanishing Criterion. Section~\ref{algorithm} gives our algorithms for computing the reflex field, the typenorm and the ideal class group isomorphic to the Galois group $Gal(CM_{\frakm}(K^r)/K^r)$ via class field theory. Section~\ref{Paragraph:Computing_Galois_Conjugates} contains our main results on the Galois conjugates of CM points and formulae for their invariants. Finally, Section~\ref{implementation} describes our implementation in Sage and shows a toy example that we computed.

\paragraph{Acknowledgement.} The first author is grateful to Jeroen Sijling for many helpful discussions. The second author thanks Christelle Vincent for preliminary discussions which led to this research. The authors acknowledge financial support from the FACE foundation.

%% file: background_ants.tex
\section{Background}\label{background}

This section briefly recalls the necessary background and notation on complex abelian varieties, theta functions and the Vanishing Criterion which fully characterizes hyperelliptic principally polarized abelian varieties. We also define the invariants of hyperelliptic curves that we will be computing in the next sections. 

\subsection{Principally polarized abelian varieties over $\C$ and period matrices}\label{Sec:periodmats}
 Let $A=\C^g/\Lambda$, with $\Lambda$ a full lattice in $\C^g$ and $E$ a Riemann form for $(\C^g,\Lambda)$. 
 A principally polarized abelian variety defined over $\C$ is isomorphic to a complex torus admitting a Riemann 
 form (~\cite[Ch. 1]{Mumford}). Therefore, we will write $(A,E)$ to denote a p.p.a.v. over $\C$. We consider 
 a \emph{symplectic} basis for the lattice $\Lambda$, by which we mean the action of $E$ on $\Lambda$ with 
 respect to this basis is given by the matrix
\begin{equation}\label{eq:omega}
\begin{split}
    J_g =
    \left( \begin{matrix}
    0 & I_g \\
    -I_g & 0
    \end{matrix} \right),
  \end{split}
\end{equation}
where $I_g$ is the $g \times g$ identity matrix.

Let $\Omega = [\Omega_1\  |\  \Omega_2]$ be the $g\times 2g$ matrix whose columns are the elements of this symplectic basis. Then the columns of $\Omega_2$ form a $\C$-basis of $\C^g$. Writing the columns of $\Omega_1$ in terms of this basis, we obtain a $g \times g$ matrix $Z$ called a \emph{period matrix}, i.e. an element of the Siegel upper half-space 
\[\mathcal{H}_g = \{ Z \in \mathcal{M}_{g}(\mathbb{C}) : Z^{T} = Z,\ \im(Z) >0 \}.\]
We note that the lattice $\Lambda$ can be written as $\Z^g + Z\Z^g$. Further to mark up the relationship between the lattice $\Lambda$ and the period matrix $Z$ discussed above, we denote this by $\Lambda_Z$. 

There is an action on $\mathcal{H}_g$ by the symplectic group 
\[ \text{Sp}_{2g}(\Z) = \{M\in \text{GL}_{2g}(\Z) : M^T J_g M = J_g\},\]
where $J_g$ is as in Equation~\eqref{eq:omega},
given by 
\begin{equation}\label{Def:Action_Sp6Z_on_H3}
M = \left( \begin{matrix}
    A & B \\
    C & D
    \end{matrix} \right): Z\longmapsto M.Z = (AZ + B)(CZ + D)^{-1},
\end{equation}
where on the right hand side multiplication is the usual matrix multiplication.

The association of $Z$ to $(\C^g/(\Z^g + Z\Z^g),E)$ gives a bijection between 
$\text{Sp}_{2g}(\Z)\backslash \mathcal{H}_g$ and the set of p.p.a.v. over $\C$ up to isomorphism. Note that by acting the period matrices with matrices in the symplectic group of level 2
\[ \Gamma_{2g}(2) = \{M\in \text{Sp}_{2g}(\Z) : M\equiv I_{2g}\pmod 2\},\]
we fix the 2-torsion on the p.p.a.v. Therefore, $\text{Sp}_{2g}(\Z)\backslash \mathcal{H}_g$ gives a bijection to the set of isomorphism classes of p.p.a.v. over $\C$ with a level 2 structure. 

\subsection{Theta functions.} 
\label{Sec: hyp}
For $\omega = (\omega_1, \omega_2) \in \C^{2g}$ and $Z \in \mathcal{H}_g$, we define the following important theta series:
\begin{equation}\label{Def: Theta_function}
\vartheta(\omega, Z) = \sum_{n \in \Z^{g}}\exp(\pi i (\omega_1 + n)^T Z (\omega_1 + n) + 2\pi i (\omega_1 + n)^ T \omega_2)).
\end{equation}

Given a period matrix $Z \in \mathcal{H}_g$, we obtain a set of coordinates on the torus $\C^g/(\mathbb{Z}^g+Z\mathbb{Z}^g)$ in the following way: A vector $x \in [0,1)^{2g}$ corresponds to the point $x_2 +  Z x_1 \in \C^g/(\mathbb{Z}^g+Z\mathbb{Z}^g)$, where $x_1$ denotes the first $g$ entries and $x_2$ denotes the last $g$ entries of the vector $x$ of length $2g$.

Of interest to us will be the values of $\vartheta(\omega, Z)$ at points $\omega \in \C^{g}$ that, 
under the natural quotient map $\C^g \to \C^g/(\mathbb{Z}^g+Z\mathbb{Z}^g)$, map to $2$-torsion points.
 These points are of the form $\omega = \xi_2 + Z \xi_1$ for $\xi = (\xi_1,\xi_2) \in (1/2)\Z^{2g}$. This motivates
  the following definition:
\begin{equation} \label{Def: Theta_function2}
\vartheta[\xi](Z) = \exp(\pi i \xi_1^T Z \xi_1 + 2 \pi i \xi_1^T \xi_2) \vartheta( \xi_2 + Z \xi_1, Z).
\end{equation}
In this context, $\xi$ is called a \emph{theta characteristic}. The value $\vartheta[\xi](Z)$ is called a \emph{theta constant}.

For $\xi, \zeta \in (1/2)\Z^{2g}$, let $e_*(\xi) = \exp(4\pi i \xi_1^T  \xi_2)$. We say that a characteristic $\xi \in (1/2)\Z^{2g}$ is \emph{even} if $e_*(\xi) = 1$ and \emph{odd} if $e_*(\xi) = -1$. If $\xi$ is even we call $\vartheta[\xi](Z)$ an \emph{even theta constant} and if $\xi$ is odd we call $\vartheta[\xi](Z)$ an \emph{odd theta constant}.

It can be easily shown that all odd theta constants vanish. We note that there are exactly $36$ even classes in $(1/2)\Z^6/\Z^6$. 
We recall there is an action of the symplectic group $\text{Sp}_{2g}(\Z)$ on the theta characteristics $\xi\in(1/2)\Z^{2g}$ defined by:
\begin{equation} 
\label{Def: Action_on_theta_characteristic}
M.\xi = M^*\xi  + \frac{1}{2}\delta_0,
\end{equation}
with $\displaystyle \label{Def:gamma_*} 
	M = \left( \begin{matrix}
    	A & B \\
    	C & D
    	\end{matrix} \right)\in \text{Sp}_{2g}(\Z)$ and $M^*=(M^{-1})^T$. 
We denote by $\displaystyle \delta_0 = \left( \begin{matrix}
   	(C^TD)_0\\ 
	(A^TB)_0
    	\end{matrix} \right)$ a column vector  where $(C^TD)_0$ and $(A^TB)_0$ are the diagonal vectors of $C^TD$ and $A^TB$, respectively. 
In this context, given a period matrix $Z\in \calH_g$, we also recall the transformation formula  on the theta constants:
\begin{equation}\label{Def:Transormation_formula_theta_serie}
\vartheta[M.\xi](M.Z) = \zeta(M)\exp(\phi(M, \xi)) \sqrt{\det(CZ + D)}~\vartheta[\xi](Z),
\end{equation}
where
\begin{enumerate}
\item $\zeta(M)$ is an eighth root of unity, depending on $M$, and
\item \label{Def:4_th_root_of_unity_in_theta_action}
$\displaystyle \phi(M, \xi) = -\pi i\left(\xi_1^T B^TD\xi_1 + \xi_2^TA^TC\xi_2 - 2 \xi_1^TB^TC\xi_2 - 2\left( D\xi_1 - C\xi_2\right)^T \left( A^TB\right)_0\right)$.
\end{enumerate}

Finally, we will also use the following group:
\[ \Gamma_{2g}(1,2) = \{M\in \text{Sp}_{2g}(\Z) : \delta_0\equiv 0\pmod 2\}.\]

\subsection{Computing the Rosenhain invariants}\label{Chapter:Computing_approximation_model}

Let $\calA_g$ be the moduli space of principally polarized abelian varieties of dimension $g$ over $\C$. By a theorem of Torelli \cite[Thm. 12.1]{Milne_AV}, there is an injective map $\calM_g\hookrightarrow \calA_g $. Inside $\calM_g$ we further restrict our attention to the subspace of hyperelliptic curves $\calM_g^{hyp}$. We will be interested in reconstructing a moduli point in $\calM_g^{hyp}$ from a point in $\calA_g$, whenever this point is in the image of $\calM_g^{hyp}\hookrightarrow \calA_g$. 

Let $\displaystyle X:y^2 = \prod^{2g+2}_{i=1}(x-\lambda_i)$ be a hyperelliptic curve of genus $g$ with $\lambda_i\in \C$. We identify the $\lambda_i$'s with the branch points for the covering map $\pi \colon X \longrightarrow \mathbb{P}^1(\C)$: $P_i=(\lambda_i,0,1)$, $i=\overline{1,2g+1}$ and $P_{\infty}=( \lambda_{2g+2},0,1)$. Let $B=\{1,2,\ldots, 2g+1,\infty\}$.
\begin{definition}\label{Def:Marked_Curves}
By a \emph{marked hyperelliptic curve} $X$ of genus $g$ we understand a pair $(X,\phi)$, where $X$ is a hyperelliptic curve $X$ and $\phi:B\rightarrow \{P_1,P_2,\ldots,P_{\infty}\}$ is a bijection. 
\end{definition}

\noindent
Given a marked hyperelliptic curve, we will refer to the bijection $\phi$ as a labeling of the branch points. We define the moduli space of marked hyperelliptic curves of genus $3$ with level $2$ structure by 
\begin{equation}\label{hypermodulispace}
\calM^{hyp}_3[2] \cong 
 \{\lambda = (\lambda_1,\ldots,\lambda_{2g+2}), \lambda_i\in\C^*\}
/PGL(2,\C),
\end{equation} 
where $PGL(2,\C)$ is the projective linear group of $\mathbb{P}^1(\C)$. 
Considering the forgetful morphism $\displaystyle \calM_g^{hyp}[2] \longrightarrow \calM_g^{hyp}$, above every point $X\in \calM_g^{hyp}$, there are exactly $(2g+2)!$ points given by tuples $(\lambda_1,\ldots,\lambda_{2g+2})$. By taking $P_{\infty}$ to be $\infty$, we can further write
\begin{equation}\label{hypermodulispace}
\calM^{hyp}_3[2] \cong 
 \{\lambda = (\lambda_1,\ldots,\lambda_{2g+1}), \lambda_i\in\C^*\}
/\lambda_i\rightarrow \gamma \lambda_i+\mu.
\end{equation} 
The unordered crossratios $\frac{\lambda_m-\lambda_l}{\lambda_m-\lambda_k}$, for $k,l,m \in B\setminus \{\infty\}$ are then true invariants for a hyperelliptic point in $\calM_g^{hyp}[2]$~\cite{Mumford}.

Let now $(A,E)$ be a hyperelliptic point in $\calA_g$. As explained in the literature, for example \cite{Birkenhake}, a choice of period matrix $Z \in \mathcal{H}_g$ for $A$ is equivalent to a choice of symplectic basis, $A_i$, $B_i$, for the homology group $H_1(X,\Z)$ of the curve. Indeed, without any further choice, there exists a unique basis $\omega_i$ of holomorphic differentials on $X$ such that $\int_{A_i} \omega_i = 1$ and $\int_{A_i} \omega_j = 0, \, i \neq j$. Then $Z$ is the matrix of the form $Z = \left(\int_{B_i} \omega_j\right)_{\substack{1\leq i,j\leq g}}.$ Conversely, any period matrix is obtained in this manner. 
This yields the Jacobi map:
\begin{equation}\label{Def:Abel_Jacobi_map}
AJ \colon \Jac(X)  \longrightarrow \C^g/\Lambda_Z, \qquad
\sum_{k=1}^s R_k - \sum_{k=1}^s Q_k  \longmapsto \left(\sum_{k=1}^s \int_{Q_k}^{R_k} \omega_i\right)_i,
\end{equation}
which is well-defined since the value of each path integral on $X$ is well-defined up to the value of integrating the differentials $\omega_i$ along the basis elements $A_i$, $B_i$, and thus up to elements of $\Lambda_Z$.

\noindent
Given a marked hyperelliptic curve $X$, we can fix a set of 2-torsion points on $A$. For this, we identify $X$ with its image in $J(X)$ by the canonical map $\iota: X \hookrightarrow J(X)$. Then the branch points $P_j$, $j=1,\ldots, 2g+2$ correspond on $J(X)$ to points of the form $e_i=[(P_j) - (P_\infty)]$.
This gives rise to a set of characteristics $\eta=(\eta_i)_{1\leq i\leq 2g+2}$ in $(1/2)\mathbb{Z}^{2g}$ such that $AJ(e_i)=(\eta_i)_2 + Z (\eta_i)_1$. Two sets $\eta'$ and $\eta''$ are said to be in the same \emph{class} if they are equal as elements in $(1/2) \mathbb{Z}^{2g}/\mathbb{Z}^{2g}$. Following Poor~\cite{poor}, we call the set $\{\eta_1,\ldots,\eta_{2g+2}\}$ an \emph{azygetic system  associated to the period matrix $Z$}. By an azygetic system for the vector space $V=(1/2) \mathbb{Z}^{2g}/\mathbb{Z}^{2g}$ we understand a set $\eta=\left\{\eta_1,\ldots,\eta_{2g+2}:~\eta_i\in V\right\}$ with the properties
\begin{equation}\label{Def:Azygetic_base}
V = \spn(\eta_i),~\sum_{i=1}^{2g+1} \eta_i = 0, \eta_{2g+2}=0, ~\text{and}~e_2(\eta_i,\eta_j) = -1,
\end{equation}
for $i,j$ different from $2g+2$ and $i\neq j$. Here, $e_2(\eta_i,\eta_j)=exp(4\pi i\eta_i^TJ_g\eta_j)$) is the Weil pairing on the 2-torsion of the complex abelian variety. For a given azygetic system, Poor defines the set $\mathcal{U}_{\eta}$ to be the set of indexes $i\in B$ such that $\eta_i$ is even.

If the marking of the curve is changed, then we act on the homology basis by a matrix $M\in \text{Sp}_{2g}(\Z)$ and a new period matrix is obtained using the construction above. This period matrix is $Z'=M.Z$ and the azygetic system associated to it is $\eta'=M^{*}\eta$. Poor shows that the action of $\text{Sp}_{2g}(\F_2)=\text{Sp}_{2g}(\Z)/\Gamma(2)$ on azygetic systems derived in this way is free and transitive.
Since there are $(2g+2)!$ different ways to label the $2g+2$ branch points of a hyperelliptic curve $X$ of genus $g$, there are several ways to assign an azygetic system to a matrix $Z \in \mathcal{H}_g$.

The Vanishing Criterion, which is due to Mumford~\cite{Mumford} and generalized by Poor~\cite{poor}, gives a characterization of hyperelliptic period matrices in terms of their associated azygetic system and theta constants. We recall the Vanishing Criterion for genus 3 and show that up to a labeling of the branch points, the eta map associated to a period matrix is unique. 

\begin{theorem}[\cite{poor}]\label{VanishingCriterion}
Let $Z\in \mathcal{H}_3$ and $\eta$ an azygetic system. The following two statements are equivalent:
 \begin{enumerate} 
 \item $Z$ is the period matrix of a simple abelian variety and the unique even theta constant of $Z$  which vanishes is $\vartheta[\eta_{\mathcal{U}_{\eta}}](Z)$.
 \item There is a marked hyperelliptic curve of genus 3 whose Jacobian has period matrix $Z$ and $\eta$ is the unique azygetic system associated to the marked curve.
 \end{enumerate}  
\end{theorem}

\begin{proof}
The fact that the Vanishing Criterion is satisfied if and only if $\eta_{\mathcal{U}_{\eta}}$ vanishes at $Z$ is proven in~\cite{BILV}. Recall that there are exactly 36 even theta characteristics in genus 3. Since $\text{Sp}_6(\F_2)$ acts freely and transitively on azygetic systems, there are $\#\text{Sp}_6(\F_2)=1451520$ classes of azygetic systems. By acting with the permutation group $S_8$ on $\Xi_3$, we get 36 classes of eta maps. Hence, if the labeling of the branch points is ignored, exactly one of these classes will realize the vanishing theta characteristic for $Z$. 
\end{proof}  


In other words, Theorem~\ref{VanishingCriterion} shows that given a hyperelliptic period matrix $Z\in \Gamma(2)\backslash \mathcal{H}_3$, choosing one of its associated azygetic systems $\eta$ fixes a labeling on the branch points, yielding a marked hyperelliptic curve. We recover the invariants of a point in $\calM_g^{hyp}[2]$ using Takase's formulae~\cite{Takase,BILV}\footnote{In~\cite{BILV}, the authors omitted to compute the sign in this formula. We give the correct formula here.}, which we recall below. 

\begin{theorem}[Takase \cite{Takase,BILV}]\label{Th:Takase}
\label{Takase_formula}
Let $Z\in \Gamma(2)\backslash \mathcal{H}_{g}$ a period matrix and $\eta\in \Xi_g$ be such that the Vanishing Criterion is satisfied. Then with notation as above, for any disjoint decomposition $B-\{\infty\}=\calV\sqcup \calW \sqcup \{k,l,m\}$ with $\# \calV=\# \calW=g-1$, we have
\begin{align*}\label{Rosenhains_Computation}
\frac{\lambda_m-\lambda_l}{\lambda_m-\lambda_k}= (-1)^{\langle (\eta_l)_1,(\eta_k)_2+(\eta_m)_2\rangle}  \left ( \frac{\vartheta[\eta_{\calU\circ(\calV\cup \{m,l\})}]\cdot 
\vartheta[\eta_{\calU\circ (\calW\cup \{m,l\})}]}{\vartheta[\eta_{\calU\circ(\calV\cup \{k,m\})}]\cdot \vartheta[\eta_{\calU\circ (\calW\cup \{k,m\})}]} (Z)\right )^2.
\end{align*}
\end{theorem}

Finally, note that by considering an affine map of $\C$ such that $f(\lambda_1) = 0, f(\lambda_2) = 1$, we may assume without restricting the generality that $X$ is given by  
\begin{eqnarray*}
X: y^2 = x(x-1)\prod^{2g-1}_{i=1}(x-\lambda_i).
\end{eqnarray*}
In this case, we say that $X$ is in \emph{normalized Rosenhain form}. The moduli space described in Equation~\ref{hypermodulispace} writes as
\begin{equation*}
\calM^{hyp}_3[2] \cong 
 \{\lambda = (\lambda_1,\ldots,\lambda_{2g-1}), \lambda_i\in\C^*-\{0,1,\infty\}, \lambda_i\neq\lambda_j\}.
\end{equation*} 
The coefficients $\lambda_i\in \C$, $\lambda_i\neq 0,1$, are called \textit{the Rosenhain invariants} of the curve and will be the focus of our work. 

\subsection{Shioda invariants}
\input{Shiodas_ants}

%% file: Shiodas_ants.tex
Shioda~\cite{Shioda} gave a set of generators for the algebra of invariants of binary octavics over the complex numbers, which are now called \emph{Shioda invariants}. In addition, over the complex numbers Shioda invariants completely determine points in $\calM^{hyp}_3$. More specifically, the Shioda invariants are $9$ weighted projective invariants $(J_2, J_3, J_4, J_5, J_6, J_7, J_8,J_9, J_{10})$, 
where $J_i$ has degree $i$, and $J_2, \ldots, J_7$ are algebraically independent, while $J_8,J_9,J_{10}$ depend algebraically on the previous Shioda invariants.

Using Igusa's map between the graded ring of Siegel modular forms of degree 3, and the graded ring of invariants of binary octavics, Lorenzo Garcia ~\cite{Elisa} proposes a set of invariants which write as quotients of modular forms. These invariants involve large powers of the modular form $\chi_{28}$ in the denominators and we do not use them for experiments since they would need too much precision to compute. 

Starting from the projective invariants $J_i$, we consider the following absolute\footnote{An absolute invariant is a ratio  of homogeneous invariants of the same degree.} Shioda invariants :
\begin{align}\label{Absolute_Shiodas}
\operatorname{Shi}^{\text{\tiny abs}}(X)&= \left( \frac{J_2^7}{\Delta}, \frac{J_2^4J_3^{2}}{\Delta}, \frac{J_2^5J_4}{\Delta}, \frac{J_5J_9}{\Delta}, \frac{J_2^4J_6}{\Delta}, \frac{J_7^2}{\Delta}, \frac{J_2^3J_8}{\Delta}, \frac{J_2^5J_9^2}{\Delta^2}, \frac{J_2^{2}J_{10}}{\Delta}\right),
\end{align}  
with $\Delta$ the discriminant which is an invariant of degree $14$ (Section $1.5$ in \cite{LerRit}). They are optimal for computations in the sense that they involve invariants of small weight and their denominator is given by primes of bad reduction (see~\cite{IKL}). Note that a subset of this set was already used by Weng~\cite{Weng} for computing models of hyperelliptic curves with CM by a field which contains $i$.
\begin{proposition}
The invariants in Equation~\ref{Absolute_Shiodas} are modular, i.e. they write as quotients of modular forms of level 1.
\begin{proof}
The proof of this statement is straightforward, by using Igusa's map on the set of invariants described by Tsuyumine~\cite{Tsuyumine1} and the relations between Tsuyumine's invariants and the Shioda projective invariants, computed by Lorenzo Garcia~\cite{Elisa}. We have
\begin{align*}
\frac{J_2^7}{\Delta} &= c_2^7 \frac{ I_2}{\Delta} = c_2^7  \rho\left( \frac{\gamma_{20}^7}{\chi_{28}^5} \right),\\
\frac{J_2^4J_3^{2}}{\Delta} &= c_2^4 c_3^2  \frac{I_2^4 I_3^2}{\Delta} = c_2^4 c_3^2  \rho\left( \frac{\gamma_{20}^4  \gamma_{30}^2}{\chi_{28}^5}\right),\\
\frac{J_2^5J_4}{\Delta} &=  \frac{c_2^5 I_2^5  (c_{41}  I_2^2 + c_{42} I_4)}{\Delta} 
= d_1  \rho \left( \frac{\gamma_{20}^7}{\chi_{28}^5} \right) + d_2  \rho\left( \frac{\gamma_{20}^5 \alpha_{12}}{\chi_{28}^4} \right),\\
\frac{J_5J_9}{\Delta} &=  \frac{\left(c_{51}  I_2 I_3 + c_{52}  I_5 \right)  \left( c_{91} I_2^3 I_3 + c_{92} I_2^5 I_5 + c_{93} I_2 I_3 I_4 + c_{94} I_2 I_7 + c_{95} I_3^3 + c_{96} I_3 I_6 + c_{97} I_4 I_5 + c_{98} I_9 \right) }{\Delta} \\
& = e_1   \rho \left( \frac{\gamma_{20}^4  \gamma_{30}^2}{\chi_{28}^5}\right) + e_2   \rho \left( \frac{\gamma_{20}^3  \gamma_{30}  \beta_{22}}{\chi_{28}^4}\right) + e_3  \rho\left( \frac{\gamma_{20}^2 \gamma_{30}^2  \alpha_{12}}{\chi_{28}^4}\right) + e_4  \rho\left( \frac{\gamma_{20}^2  \gamma_{30}  \beta_{14}}{\chi_{28}^3}\right) + e_5  \rho\left( \frac{\gamma_{20}  \gamma_{30}^4}{\chi_{28}^5}\right) \\
& +  e_6  \rho \left( \frac{\gamma_{20}  \gamma_{30}^2  \alpha_4}{\chi_{28}^3}\right) + e_7  \rho \left( \frac{\gamma_{20}  \gamma_{30}  \alpha_{12}  \beta_{22}}{\chi_{28}^3}\right) + e_8  \rho \left( \frac{\gamma_{20}  \gamma_{30}  \alpha_{6}}{\chi_{28}^2}\right) + e_9  \rho \left( \frac{\gamma_{20}^3 \gamma_{30}  \beta_{22}}{\chi_{28}^4}\right) +  e_{10}  \rho \left( \frac{\gamma_{20}^2   \beta_{22}^2}{\chi_{28}^3}\right) \\
&  + e_{11}  \rho \left( \frac{\gamma_{20}  \beta_{22}  \beta_{14}}{\chi_{28}^2}\right) + e_{12}  \rho \left( \frac{\gamma_{30}^3  \beta_{22}}{\chi_{28}^4}\right) +  e_{13}  \rho \left( \frac{\gamma_{30}  \beta_{22}  \alpha_4}{\chi_{28}^2}\right) + e_{14}  \rho \left( \frac{\alpha_{12}  \beta_{22}^2}{\chi_{28}^2}\right)  + e_{15}  \rho \left( \frac{\beta_{22}  \alpha_{6}}{\chi_{28}}\right),\\
\frac{J_2^4J_6}{\Delta} & =  \frac{ c_2^4  I_2^4  \left( c_{61}  I_2^3 + c_{62} I_2 I_4 + c_{63} I_3^2 + c_{64}  I_6 \right)}{\Delta} \\
& =  f_1  \rho \left( \frac{\gamma_{20}^7}{\chi_{28}^5} \right) + f_2  \rho \left( \frac{\gamma_{20}^5  \alpha_{12}}{\chi_{28}^4} \right) + f_3  \rho \left( \frac{\gamma_{20}^4  \gamma_{30}^2}{\chi_{28}^5} \right) + f_4  \rho \left( \frac{\gamma_{20}^4  \alpha_4}{\chi_{28}^3} \right),\\
\frac{J_7^2}{\Delta} & =  \frac{ \left( c_{71}  I_2^2  I_3 + c_{72}  I_2 I_5 + c_{73}  I_3  I_4 + c_{74}  I_7 \right)^2 }{\Delta} \\
& = g_{1}  \rho\left( \frac{\gamma_{20}^4 \gamma_{30}^2}{\chi_{28}^5} \right) + g_{2}  \rho\left( \frac{\gamma_{20}^3 \gamma_{30} \beta_{22}}{\chi_{28}^4} \right) + g_{3}  \rho\left( \frac{\gamma_{20}^2 \gamma_{30}^2 \alpha_{12}}{\chi_{28}^4} \right) + g_{4}  \rho\left( \frac{\gamma_{20}^2 \gamma_{30} \beta_{14}}{\chi_{28}^3} \right) + g_{5}  \rho\left( \frac{\gamma_{20}^2 \beta_{22}^2}{\chi_{28}^3} \right)\\
&  + g_{6}  \rho\left( \frac{\gamma_{20} \gamma_{30} \alpha_{12} \beta_{22}}{\chi_{28}^3} \right) + g_{7}  \rho\left( \frac{\gamma_{20} \beta_{22} \beta_{14}}{\chi_{28}^2} \right) + g_{8}  \rho\left( \frac{\gamma_{30}^2 \alpha_{12}^2}{\chi_{28}^3} \right) + g_{9}  \rho\left( \frac{\gamma_{30} \alpha_{12} \beta_{14}}{\chi_{28}^2} \right) + g_{10}  \rho\left( \frac{\beta_{14}}{\chi_{28}} \right) ,\\
\frac{J_2^3J_8}{\Delta} & = \frac{c_2^3 I_2^3 \left( c_{81} I_2^4 + c_{82} I_2^2 I_4 + c_{83} I_2 I_3^2 + c_{84} I_2 I_6 + c_{85} I_3 I_5 + c_{86} I_4^2 + c_{87} I_8\right) }{\Delta}\\
& = h_1 \rho \left( \frac{\gamma_{20}^7}{\chi_{28}^5} \right) + h_2 \rho \left( \frac{\gamma_{20}^5 \alpha_{12}}{\chi_{28}^4} \right) + h_3 \rho \left( \frac{\gamma_{20}^4 \gamma_{30}^2}{\chi_{28}^5} \right) + h_4 \rho \left( \frac{\gamma_{20}^4 \alpha_4}{\chi_{28}^3} \right)\\
& + h_5  \rho \left( \frac{\gamma_{20}^3 \gamma_{30} \beta_{22}}{\chi_{28}^4} \right) + h_6 \rho \left( \frac{\gamma_{20}^3 \alpha_{12}^2}{\chi_{28}^3} \right) + h_7 \rho \left( \frac{\gamma_{20}^3 \gamma_{24}}{\chi_{28}^3} \right),\\
\frac{J_2^5J_9^2}{\Delta^2} & = \frac{c_2^5 I_2^5 \left( c_{91} I_2^3 I_3 + c_{92} I_2^2 I_5 + c_{93} I_2 I_3 I_4 + c_{94} I_2 I_7 + c_{95} I_3^3 + c_{96} I_3 I_6 + c_{97} I_4 I_5 + c_{98} I_9\right)^2}{\Delta^2}\\
& = i_{1} \rho \left( \frac{\gamma_{20}^{11} \gamma_{30}^2}{\chi_{28}^{10}}\right) + i_{2} \rho \left( \frac{\gamma_{20}^{10} \gamma_{30} \beta_{22}}{\chi_{28}^9}\right) + i_{3} \rho \left( \frac{\gamma_{20}^9 \gamma_{30}^2 \alpha_{12}}{\chi_{28}^9}\right) + i_{4} \rho \left( \frac{\gamma_{20}^9 \gamma_{30} \beta_{14}}{\chi_{28}^8}\right) + i_{5} \rho \left( \frac{\gamma_{20}^8 \gamma_{30}^4}{\chi_{28}^{10}}\right)\\
& + i_{6} \rho \left( \frac{\gamma_{20}^8 \gamma_{30}^2 \alpha_4}{\chi_{28}^8}\right) + i_{7} \rho \left( \frac{\gamma_{20}^8 \gamma_{30} \alpha_{12} \beta_{22}}{\chi_{28}^8}\right) + i_{8} \rho \left( \frac{\gamma_{20}^8 \gamma_{30} \alpha_6}{\chi_{28}^7}\right) + i_{9} \rho \left( \frac{\gamma_{20}^9 \beta_{22}^2}{\chi_{28}^8}\right) + i_{10} \rho \left( \frac{\gamma_{20}^7 \gamma_{30}^2 \alpha_{12}^2}{\chi_{28}^8}\right)\\
& + i_{11} \rho \left( \frac{\gamma_{20}^7 \beta_{14}^2}{\chi_{28}^6}\right) + i_{12} \rho \left( \frac{\gamma_{20}^5 \gamma_{30}^6}{\chi_{28}^{10}}\right) + i_{13} \rho \left( \frac{\gamma_{20}^5 \gamma_{30}^2 \alpha_4^2}{\chi_{28}^6}\right) + i_{14} \rho \left( \frac{\gamma_{20}^5 \alpha_{12}^2 \beta_{22}^2}{\chi_{28}^6}\right) + i_{15} \rho \left( \frac{\gamma_{20}^5 \alpha_6^2}{\chi_{28}^4}\right),
\end{align*}

\begin{align*}
\frac{J_2^{2}J_{10}}{\Delta} & = \frac{c_2^2 I_2^2 \left( c_{101} I_2^5 + c_{102} I_2^3 I_4 + c_{103} I_2^2 I_3^2 + c_{104} I_2^2 I_6 + c_{105} I_2 I_3 I_5 + c_{106} I_2 I_4^2 \right)}{\Delta}\\
& + \frac{ c_2^2 I_2^2  \left( c_{107} I_2 I_8 + c_{108} I_3^2 I_4 + c_{109} I_3 I_7 + c_{110} I_4 I_6 + c_{111} I_5^2 + c_{112} I_{10} \right) }{\Delta} \\
& = j_{1} \rho\left( \frac{\gamma_{20}^7}{\chi_{28}^5} \right) + j_{2} \rho\left( \frac{\gamma_{20}^5 \alpha_{12}}{\chi_{28}^4} \right) + j_{3} \rho\left( \frac{\gamma_{20}^4 \gamma_{30}^2}{\chi_{28}^5} \right) + j_{4} \rho\left( \frac{\gamma_{20}^4 \alpha_4}{\chi_{28}^3} \right)\\ 
& + j_{5} \rho\left( \frac{\gamma_{20}^3 \gamma_{30} \beta_{22}}{\chi_{28}^4} \right) + j_{6} \rho\left( \frac{\gamma_{20}^3 \alpha_{12}}{\chi_{28}^3} \right) + j_{7} \rho\left( \frac{\gamma_{20}^3 \gamma_{24}}{\chi_{28}^3} \right) + j_{8} \rho\left( \frac{\gamma_{20}^2 \gamma_{30}^2 \alpha_{12}}{\chi_{28}^4} \right)\\
& + j_{9} \rho\left( \frac{\gamma_{20}^2 \gamma_{30} \beta_{14}}{\chi_{28}^3} \right) + j_{10} \rho\left( \frac{\gamma_{20}^2 \alpha_{12} \alpha_4}{\chi_{28}^2} \right) + j_{11} \rho\left( \frac{\gamma_{20}^2 \beta_{22}^2}{\chi_{28}^3} \right) + j_{12} \rho\left( \frac{\gamma_{20}^2 \beta_{16}}{\chi_{28}^2} \right)
 \end{align*}
where the constants $c_k,d_k,e_k,f_k, g_k, h_k,i_k,j_k$ are computed in~\cite{Elisa}.
\end{proof}
\end{proposition}

%% file: algorithm_ants.tex
\section{Computing abelian varieties with CM}\label{algorithm}

In this section, we review well known definitions and results from the theory of complex multiplication, and give algorithms for the computation of several notions, such as the reflex field, the typenorm and the Galois action on CM points, with the final goal of stating an effective version of Shimura's second theorem of complex multiplication. All these algorithms were implemented in the computer algebra system SAGE~\cite{sagemath} and are called during the computation of class polynomials in Algorithm~\ref{Compute_Shimura_Reciprocity}. 

\subsection{Reflex field computation}\label{chapter3}
Let $K/\Q$ be a CM field and let $L$ be the Galois closure of $K$ with Galois group $Gal(L/\Q)$.
A \emph{CM type} of $K$ is a set $\Phi=\{\phi_1, \ldots, \phi_g\}$ of $g$ embeddings $K \into \C$ such that no 
two embeddings appearing in $\Phi$ are complex conjugates. We say that $\Phi$ is \emph{induced} from a CM subfield 
$K_0$ of $K$ if the set $\{ \phi|_{K_0}: \phi\in\Phi\}$ is a CM type of $K_0$.  A CM type of $K$ is called \emph{primitive} if it is not induced by a proper CM subfield $K'\subset K$. In this paper, we fix the tupel $\displaystyle\left( K, \Phi\right)$ and call it a \emph{CM-pair}. Since $L$ is a CM field (\cite[Cor. 1.5]{Milne_CM}), $\Phi$ extends to a CM type $\Phi_L$ of $L$, namely by 
\begin{equation}\label{Extended_CM_Type} 
\Phi_L = \{\phi: L \rightarrow \C:\ \phi|_{K} \in \Phi\}.
\end{equation}

We fix once and for all an embedding $\iota_K: K \rightarrow L$ and an embedding $\pi: L \rightarrow \C$. 
With these in hand, we identify elements in $\Phi_L$ with elements of the automorphism group $Gal(L/\Q)$ by associating to every $\phi\in\Phi$ an element $\sigma \in Gal(L/\Q)$ such that the following diagram commutes:
\begin{equation}
\label{Embeddings_Kr_into_L}
\begin{tikzcd}
   L \arrow[r,dashed, "\sigma"] & L \arrow[d, "\pi"]  \\
  K \arrow[u, "\rotatebox{90}{$\iota_{\tiny{K}}$}"] \arrow[r, "\phi"] &\C
\end{tikzcd}
\end{equation}
Note that this identification is certainly dependent on the embeddings $\iota_K$ and  $\pi$.

Let \begin{math} \label{Lifted_Reflex_Type}
\Phi_L^{-1} = \{\pi\circ \sigma^{-1}\in \Aut(L): \phi = \pi\circ \sigma\ \text{for}\ \phi\in \Phi_L\}. 
\end{math} One can easily check that $\Phi_L^{-1}$ is a CM type on $L$ if and only if $\Phi_L$ is a CM type on $L$. 
 We denote by $H^r$ the subgroup of $\Gal(L/\Q)$ of the form  
\begin{equation}\label{Reflex_Field_}
H^r  = \{\sigma\in \Gal(L/\Q): \Phi^{-1}_L \sigma= \Phi^{-1}_L\} = \{\sigma\in \Gal(L/\Q): \sigma\Phi_L = \Phi_L \}.
\end{equation} 
\begin{definition}\label{Reflex_Field}
The subfield of $L$ fixed by the the group $H^r$ in Equation~$\eqref{Reflex_Field_}$ is called the \emph{reflex field} of $(K, \Phi)$. We denote it by $K^r$. 
\end{definition}
Note that, from a computational point of view, choosing $K^r$ as the field fixed by $H^r$ also means fixing the embedding $\iota_{K^r}: K^r \rightarrow L$.

As shown for instance in~\cite[Prop. 1.18]{Milne_CM}, $K^r$ is also a CM field and the associated CM type to $K^r$ is given by the following construction:
\begin{equation}\label{Reflex_Type}
\Phi^r = \Phi^{-1}_L|_{K^r} = \{\phi|_{K^r}: \phi \in \Phi^{-1}_L\}.
\end{equation}   
We call the tuple $(K^r, \Phi^r)$ the \emph{reflex CM-pair} of $(K,\Phi)$. We summarize briefly our procedure for computing $(K^r, \Phi^r)$ based on Definition~\ref{Reflex_Field} in Algorithm~\ref{Compute_Reflex}, in Appendix~\ref{algorithms}.

\subsection{The reflex typenorm}\label{image_typenorm}
Let $\displaystyle (K, \Phi)$ be a primitive CM-pair with Galois closure $\displaystyle L$ of $K$ and reflex CM-pair $\displaystyle (K^r,\Phi^r)$. The \emph{reflex typenorm} is the map 
\begin{equation}\label{reflex_type_norm_formula}
\begin{split} 
N_{\Phi^r}:\ K^r\rightarrow K   \subset L,\ 
 x \mapsto \prod_{\phi\in\Phi^r} \phi(x).
\end{split}
\end{equation}
Since $(K^r)^r=K$, it follows that $K$ is fixed by the group $\displaystyle H = \{ \sigma \in \Gal(L/\Q): \sigma \Phi^r_L = \Phi^r_L\}$, one can easily check that $\displaystyle N_{\Phi^r}$ lies in $K$.
\begin{lemma}\label{Reflex_type_norm_on_ideals}
The reflex typenorm in Equation~$\eqref{reflex_type_norm_formula}$ 
induces a map between ideals and 
\begin{eqnarray*}
  N_{\Phi^r}:  I(K^r) \rightarrow I(K),~
  \mathfrak{a} \mapsto \prod_{\phi\in\Phi^r} \phi(\mathfrak{a}).
\end{eqnarray*} 
which extends to a homomorphism between class groups~$N_{\Phi^r}:  Cl(K^r) \rightarrow Cl(K)$.
\begin{proof}
Let $\fraka$ be an ideal in the reflex field $K^r$. To prove the Lemma we follow the idea from M.Streng \cite{Streng_RTypeNorm} and prove the existence of an ideal $\fraka'$ in $K$ such that its lift $\fraka'\calO_L$ corresponds to the ideal in the reflex typenorm of $\fraka$. With other words we have to show the equality 
\begin{equation*}\label{prove_ideal_equality}
\fraka' \calO_L = N_{\Phi^r}(\fraka)\calO_L.
\end{equation*}
To prove the equality above we use some standard constructions in combination with the representation of $1$ by the gcd. Let $\alpha, \gamma$ be elements in $\fraka$ such that 
\begin{eqnarray*} 
1 = \gcd\left( \gamma\fraka^{-1}, \left( N_{K^r/\Q}(\alpha) \right) \right).
\end{eqnarray*}
By construction every $\phi\in\Phi^r$ correspond to an $L$-automorphism which does not change the value of the norm of $\alpha$. By replacing the sign of the norm by its definition we write the $\gcd$ as 
\begin{eqnarray*}
1 = \gcd\left( \prod_{\phi\in\Phi^r}\phi(\gamma)\phi(\fraka^{-1}), \left( \prod_{\sigma\in Gal(L/\Q)} \alpha^\sigma \right)  \right).
\end{eqnarray*} 
Let now $\delta,\beta$ be elements in $K$, and $\frakc^{-1}$ an fractional ideal in $K$ of the form
\begin{eqnarray*}
\delta = \prod_{\phi\in\Phi^r} \phi(\gamma), ~ \beta = \prod_{\phi\in \Phi^r} \phi(\alpha), ~ \frakc^{-1} = \prod_{\phi\in\Phi^r}\phi(\fraka^{-1}).
\end{eqnarray*}
By construction the elements $\delta, \beta$ belong to $K$, since our CM types are primitive and the CM field $K$ is fixed by the group $H$ above. We claim that  
\begin{eqnarray*} 
(1)\calO_L = \delta\frakc^{-1} + \beta\calO_L.
\end{eqnarray*} 
This equality follows immediately from the above representation of $1$ by the $\gcd$. Multiplication by $\frakc$ provides an equation of the form
\begin{eqnarray*}
\frakc = \frakc\beta\calO_L + \gamma.
\end{eqnarray*}
Let now $\fraka' = \beta\calO_K + \delta\calO_K$ be the $\calO_K$-ideal and consider its lift $\fraka'\calO_L = \beta\calO_L + \gamma\calO_L$. We claim that $\frakc = \fraka'\calO_L$. First $\frakc \subset \fraka'\calO_L$ which follows immediately from the inclusion $\frakc\beta\calO_L \subset \beta\calO_L$, and second $\fraka'\calO_L \subset \frakc$ since $\beta\in\frakc$ after construction. Further since principal ideals map to principal ideals via the reflex typenorm, we conclude the properties of the induced map on the class groups.
\end{proof}
\end{lemma}
When computing the typenorm of an ideal $\mathfrak{a}$, the product $\prod_{\phi\in\Phi^r} \phi(\mathfrak{a})$ gives à priori an ideal in $L$. To identify the ideal in $K$ lying below this ideal, we first compute the factorization of this ideal and rely on an algorithm in~\cite[Algorithm 2.5.3]{Cohen} to get the prime ideal lying below each of the ideals appearing in this factorization. The Algorithm \ref{Type_Norm_Computation} in Appendix~\ref{algorithms} describes briefly the computation of the reflex typenorm.

\subsection{Class field theory}
For a number field $\displaystyle K$ and a (finite) modulus $\displaystyle \mathfrak{m}$, i.e. a finite product of prime ideals in $K$, let $\displaystyle I_{\frakm}(K)$ be the group of all fractional $\displaystyle \cO_{K}$ ideals coprime to $\displaystyle \mathfrak{m}$, with subgroup 
\begin{equation*} 
\begin{split}
P_{\frakm}(K) & = \left\{ \mathfrak{a}\in I_{\frakm}(K): \mathfrak{a} = \alpha\cO_{F},\  \alpha \equiv 1\ (\bmod^* \mathfrak{m})\right\},
\end{split}
\end{equation*}
where the congruence $\displaystyle \alpha \equiv 1\ (\bmod^*\mathfrak{m})$ means that for all primes $\displaystyle \mathfrak{p}$ appearing in the factorisation of $\displaystyle \mathfrak{m}$ 
we have $\displaystyle \nu_\mathfrak{p}(\alpha - 1) \ge \nu_\mathfrak{p}(\mathfrak{m})$. The \emph{ray class group} of $K$ for the modulus $\mathfrak{m}$ is defined as the quotient group $Cl_\mathfrak{m}(K) = I_{\frakm}(K)/P_{\frakm}(K)$.

For a modulus $\displaystyle\mathfrak{m}$ in $K$ there is a unique Abelian extension $\displaystyle /K$, denoted by $\displaystyle \mathcal{H}(\frak{m})$, all of whose ramified primes divide $\displaystyle \mathfrak{m}$ s.t. the  kernel of the \emph{Artin Map}
\begin{equation*}
\displaystyle \Phi_{\mathfrak{m}}: I_{K}(\mathfrak{m}) \rightarrow Gal(\mathcal{H}_\mathfrak{m}/K)
\end{equation*}
is equal to $\displaystyle P_K(\mathfrak{m})$. The field $\displaystyle \mathcal{H}(\frak{m})$ is called the \emph{ray class field} of $K$ of modulus $\frakm$ (see for instance~\cite[Theorem 8.6.]{Cox}). If $\frakm = (1)$ then $Cl_\frakm(K)$ is the ordinary ideal class group $Cl(K)$ and $\mathcal{H}(1)$ is the \emph{Hilbert class field} $H$ of $K$. 

Let $\displaystyle (K, \Phi)$ be a primitive CM-pair with Galois closure $\displaystyle L$ and reflex CM-pair $\displaystyle (K^r,\Phi^r)$. Applying the reflex typenorm $N_{\Phi^r}$ on fractional ideals $\mathfrak{a} \in I_{K^r}(\mathfrak{m})$ coprime to $\frakm$ yields a map 
$N_{\Phi^r}:I_\mathfrak{m}(K^r) \rightarrow I_\mathfrak{m}(K)$ and induces a group homomorphism between the ray class groups
$N_{\Phi^r}: Cl_\mathfrak{m}(K^r) \rightarrow Cl_\mathfrak{m}(K)$. Now let  
\begin{equation}\label{H_Kr}
H_{\mathfrak{m}}(K^r) = \left\{ \mathfrak{a} \in I_{\mathfrak{m}}(K^r):\ \begin{split} & \exists \alpha\in K^* \ \text{with} & N_{\Phi^r}(\mathfrak{a}) = \alpha\calO_K,\\ &  N_{K/\Q}(\mathfrak{a}) = \alpha\overline{\alpha}, & \alpha \equiv 1\ (mod\ ^*\mathfrak{m})  \end{split}  \right\}
\end{equation}
and note that $\displaystyle P_{\mathfrak{m}}(K^r) \subset H_{\mathfrak{m}}(K^r)\subset I_{\mathfrak{m}}(K^r)$. Then, after \cite[Theorem 8.6.]{Cox}, there is a unique Abelian extension of $K^r$, denoted by $CM_{\frakm}(K^r)$, such that
\begin{eqnarray*}
\Gal(CM_{\frakm}(K^r)/K^r)\cong I_{\frakm}(K^r)/H_{\frakm}(K^r). 
\end{eqnarray*}

\begin{lemma}\label{Kernel_Reflex_Type_Norm}
Let $N_{\Phi^r}: Cl_\mathfrak{m}(K^r) \rightarrow Cl_\mathfrak{m}(K)$ be the reflex typenorm map. Then
\begin{itemize}  
\item[(a)] The kernel of this map is given by the quotient 
\begin{equation*} 
\ker N_{\Phi^r} = H_{\mathfrak{m}}(K^r)/P_{\mathfrak{m}}(K^r).
\end{equation*}
\item[(b)] The image of this map is isomorphic to 
$\displaystyle I_{\mathfrak{m}}(K^r)/H_{\mathfrak{m}}(K^r)$.
\end{itemize}
\end{lemma}
\begin{proof}
(a) Let us show that if $\displaystyle \mathfrak{a}\in I_\mathfrak{m}(K^r)$  with $\displaystyle \mathfrak{a}$ an element in the kernel of $N_{\Phi^r}$, then $\displaystyle \mathfrak{a}$ lies in $\displaystyle H_{\mathfrak{m}}(K^r)$. For every $\displaystyle \mathfrak{a}\in \ker N_{\Phi^r}$ there exists an element $\displaystyle \alpha\in K^*$, such that $\displaystyle N_{\Phi^r}(\mathfrak{a}) = (\alpha) \calO_K$ is principal. The reflex typenorm is a half norm, i.e. 
$\displaystyle N_{\Phi^r}(\mathfrak{a})\overline{N_{\Phi^r}(\mathfrak{a})} = \alpha\overline{\alpha} = N_{K/\Q}(\mathfrak{a})$.
The last congruence in Equation~(\ref{H_Kr}) follows from the fact that the reflex typenorm map of an ideal coprime to $\mathfrak{m}$ is coprime to $\mathfrak{m}$. It follows that $\displaystyle \mathfrak{a}\in H_{\mathfrak{m}}(K^r)$ and therefore $\displaystyle \ker N_{\Phi^r} \subset H_{\mathfrak{m}}(K^r)/P_{\mathfrak{m}}(K^r)$. The inclusion $\displaystyle H_{\mathfrak{m}}(K^r)/P_{\mathfrak{m}}(K^r) \subset \ker N_{\Phi^r}$ follows immediately from the definition of $H_{\frakm}(K^r)$.

\noindent
(b) It follows immediately from point (a): 
\begin{equation*}
\begin{split}
 N_{\Phi^r}(Cl_{\frakm}(K^r)) & \cong Cl_{\mathfrak{m}}(K^r)/\ker N_{\Phi^r}\\
& \cong \left(I_{\mathfrak{m}}(K^r)/P_{\mathfrak{m}}(K^r) \right)/\left(H_{\mathfrak{m}}(K^r)/P_{\mathfrak{m}}(K^r) \right)\\
&  \cong I_{\mathfrak{m}}(K^r)/H_{\mathfrak{m}}(K^r). 
\end{split}
\end{equation*}
\vspace*{-0.2 cm}
\end{proof}

\subsection{CM abelian varieties} 
\label{Sec: CM} Let $\calO$ be an order of the ring of integers of the CM field $K$. 
We say that an abelian variety $A$ defined over a field $k$ \emph{has CM by $\calO$} if there exists an 
embedding $\calO \into \End(A)$, where $\End(A)$ is the geometric endomorphism ring of $A$. In this 
article we focus on the case where $\End(A) \cong \calO_K$, the ring of integers of $K$. 

Let $\frakD_{K/\Q}$ be the different of $K$, and let $\fraka$ be a fractional ideal of $K$. 
Suppose that the ideal $(\frakD_{K/\Q}\fraka\overline{\fraka})^{-1}$ is principal and generated by 
$\xi\in K^{\times}$ such that $\im(\phi(\xi))>0$ for all $\phi\in\Phi$.
Then by tensoring the map 
\begin{align*}
(\Phi(\fraka), \Phi(\fraka)) &\rightarrow \Q, ~~~~~~~~~~~~~\ (\Phi(x),\Phi(x)) \mapsto \Tr_{K/\Q}(\xi\overline{x}y)
\end{align*}
with $\R$ we obtain a Riemann form $E_{\Phi,\xi} : \C^g \times \C^g \rightarrow \R$. 
Hence for any triple $(\Phi, \fraka, \xi)$ as above, the pair $(\C^g/\Phi(\fraka), E_{\Phi, \xi})$ is a 
p.p.a.v. of dimension $g$ with CM by $\calO_K$ and of type $\Phi$. We will also denote it by $A(\Phi, \fraka, \xi)$.
Conversely, every p.p.a.v. of dimension $g$ with CM by $\calO_K$ is 
isomorphic to $A(\Phi, \fraka, \xi)$ for some triple $(\Phi, \fraka, \xi)$ as above. The abelian variety 
$\calA(\Phi, \fraka, \xi)$ is \emph{simple} if and only if $\Phi$ is primitive. 
Given a primitive CM type, we denote the set of p.p.a.v. obtained in this way by $\Princ(K, \Phi)$.   
Note that to go from the triple $(\Phi, \fraka, \xi)$ to a period matrix as described in Section~\ref{Sec:periodmats}, 
it suffices to write a basis for the ideal $\fraka$ that is symplectic with respect to the Riemann form $E_{\Phi, \xi}$. 
This basis gives the matrix $\Omega$, and then the period matrix is simply $Z = \Omega_2^{-1}\Omega_1$.

Let $(A,E)$ be a p.p.a.v. with CM by $K$ and $G$ the automorphism group of $A$. We denote by $V$ the Kummer variety, i.e. the quotient variety of $A$ by $G$. Let $F:A\rightarrow V$ be the corresponding onto map.   

\begin{theorem}\cite[Main Theorem 2]{Shimura}\label{ShimuraSecondMainTheorem}
Let $(A,E)$ be a principally polarized abelian variety with complex multiplication by $K$ and CM type $\Phi$ and $(V,F)$ its Kummer variety. Let $\mathfrak{m}$ be an integral ideal of $K$ and let $m$ be the smallest integer divisible by $\mathfrak{m}$. 
Let $K^r$ be the reflex field of $K$ and $H_{\mathfrak{m}}(K^r)$ the ideal group modulo $\frakm$ defined in 
Equation~$\eqref{H_Kr}$. Let $x$ be a $\mathfrak{m}$-torsion point on $A$, $k$ the field of moduli of $A$ and $k^*$
 be the composite of $k$ and $K^r$. Then $k^*(F(x))$ is the class field corresponding to the ideal group $H_{\frakm}(K^r)$. 
\end{theorem}

\noindent
Shimura proves this theorem by showing that $\Gal(k^*(F(x))/K^r)$ is isomorphic to $I_{\mathfrak{m}}(K^r)/H_{\mathfrak{m}}(K^r)$. Moreover, the Galois conjugates of $F(x)$ are obtained by considering an action of $I_{\mathfrak{m}}(K^r)/H_{\mathfrak{m}}(K^r)$ on the set of p.p.a.v. with CM by $(K,\Phi)$ and a fixed $\frakm$-torsion point. We denote this set here by $\Princ(K,\Phi,\frakm)$. 

\begin{definition}\label{Group_Action_On_ppav}
Let $x$ be a fixed point on the abelian variety $A(\Phi, \fraka ,\xi)$. 
\begin{equation*}
\begin{alignedat}{2}
 I_{\mathfrak{m}}(K^r)/H_{\mathfrak{m}}(K^r)& \times  \Princ(K,\Phi,\frakm) && \longrightarrow \Princ(K,\Phi,\frakm),\\
 ([\frakc] , & A(\fraka,\xi,x\bmod \fraka)) && \longmapsto A(N_{\Phi^r}(\frakc)^{-1}\fraka, N_{K^r/\Q}(\frakc)\xi, x\bmod {N_{\Phi^r}(\frakc)^{-1}\fraka}).
\end{alignedat}
\end{equation*}
\end{definition}

\noindent
This action will be crucial in our computation of Galois conjugates. By examining Equation~\eqref{Group_Action_On_ppav}, we see that in actual computations, to enumerate the Galois conjugates for a CM point, one can use directly $N_{\Phi^r}(Cl_{\frakm}(K^r))$, which is isomorphic to $I_{\frakm}(K^r)/H_{\frakm}(K^r)$ by Lemma~\ref{Kernel_Reflex_Type_Norm}. To this purpose, in our implementation we obtained a set of generators for $Cl_{\frakm}(K^r)$ using MAGMA, and then implemented an algorithm for computing a subgroup from a set of elements, to get $N_{\Phi^r}(Cl_{\frakm}(K^r))$ as a subgroup of $Cl_{\frakm}(K)$.  

%% file: rosenhains_ants.tex
\section{Computing class polynomials}\label{Paragraph:Computing_Galois_Conjugates}

We turn our attention now to the computation of invariants of a hyperelliptic curve of genus $3$ with CM by $\cO_K$, and more precisely to obtaining their minimal polynomials over the reflex field. As explained in the introduction, we start by showing that given a hyperelliptic CM point with CM by $\cO_K$, all CM points obtained via the action in Equation~\eqref{Group_Action_On_ppav} are hyperelliptic. This will allow to compute the Galois conjugates of the Shioda and Rosenhain invariants, without any prior knowledge of the class fields these generate. 

We will need the following Siegel modular forms of weight $18$ and $140$ of level 1 introduced by Igusa~\cite{igusa67}:
\begin{equation}\label{eq:chi18}
\chi_{18}(Z)  = \prod_{i=1}^{36} \vartheta[\xi_i](Z),~
\Sigma_{140} (Z) = \sum_{i=1}^{36}\prod_{j\ne i} \vartheta[\xi_j](Z)^8,
\end{equation}
for all $Z\in \mathcal{H}_3$. These two modular forms allows us to fully characterize the hyperelliptic locus in the moduli space of principally polarized abelian varieties of dimension 3.

\begin{proposition}[Igusa {\cite[Lemma 11]{igusa67}}]
\label{Theorem:Igusa_single_eq_vanishing_class}
Let $Z$ be period matrix of a p.p.a.v. $A$ with polarization $E$. Then $\displaystyle (A, E)$ is isomorphic to a simple hyperelliptic Jacobian if and only if $\displaystyle \chi_{18}(Z) = 0$ and  $\displaystyle \Sigma_{140} (Z) \ne 0$.  
\end{proposition}

\begin{notation}\label{notation_matrices}
  Let $\displaystyle (\C^g/\Phi(\fraka), E_{\Phi,\xi})$ be a p.p.a.v. given by a triple $(\Phi, \fraka, \xi)$. Let $\displaystyle B = (B_1| B_2)$ is a $(6\times 3)$ complex-valued matrix containing a symplectic basis for $\Phi(\fraka)$ with respect to $E_{\Phi,\xi}$, then let $Z = B_2^{-1}B_1\in \mathcal{H}_3$ be the corresponding period matrix. Let $\mathfrak{m}=(N)$, for some integer $N$. For any $[\frakc] \in I_\frakm(K^r)/H_\frakm(K^r)$ the action described in Equation~\eqref{Group_Action_On_ppav} yields a new p.p.a.v. 
$(A^{\frakc},E^{\frakc})$ given by the triple $\displaystyle\left( \Phi, N_{\Phi^r}(\frakc)^{-1}\fraka, N_{K^r/\Q}(\frakc)\xi\right)$. In a similar manner, there is a $C = (C_1|C_2)$ containing a symplectic basis for $\displaystyle\Phi(N_{\Phi^r}(\frakc)^{-1}\fraka)$ with respect to $E_{\Phi, N_{K^r/\Q}(\frakc)\xi }$ and let $Z' = C_2^{-1}C_1\in \mathcal{H}_3$. We express $C$ in terms of $B$ by taking a matrix $M$, such that $C = BM^T$. The matrix $M$ is in $GSp_{2g}(\Q)$ and is $N$-integral and invertible $\pmod{N}$ with inverse $U\in GSp_{2g}(\Z/N\Z)$. For any modular form of level $N$, we will denote by $f^U(Z')=f(\tilde{U}Z')$, for any $\tilde{U}\in Sp_{2g}(\Z)$ a lift of $U$. Let us fix such a lift. This notation will be used all thourought this section.
\end{notation}

\begin{theorem}
Let $(A,E)$ be a simple p.p.a.v. with CM by $\calO_K$ and let $[\frakc] \in Cl(K^r)$. Then $\displaystyle (A^{\frakc}, E^{\frakc})$ is isomorphic to a hyperelliptic Jacobian if and only if $\displaystyle(A,E)$ is. Moreover, if $X$ and $X^{\frakc}$ are the corresponding hyperelliptic curves, we have the following relation between their Shioda invariants:
\begin{eqnarray}\label{Conjugate_Shiodas}
(\operatorname{Shi}^{\text{\tiny abs}}(X))_i^{\sigma}=\operatorname{Shi}^{\text{\tiny abs}}(X^{\frakc})_i,~i=\overline{2,10}. 
\end{eqnarray}
\end{theorem}
\begin{proof}
Let $Z$ and $Z'$ be the period matrices corresponding to the p.p.a.v's $\displaystyle(A,E)$ and $\displaystyle (A^{\frakc}, E^{\frakc})$. Since $A$ and $A^\frakc$ are simple, the Siegel modular form $\Sigma_{140}$ in Equation \eqref{eq:chi18} has values different from $0$ at $Z$ and $Z'$ and is holomorphic (since it writes as polynomials of theta constants). We consider the quotient $\displaystyle \frac{\chi^{70}_{18}}{\Sigma^9_{140}}$, which is a modular function of level one. We denote by $\sigma \in Gal(CM_{(1)}(K^r)/K^r)$ the element corresponding to $[\mathfrak{c}]$ under the Artin isomorphism. Then by using the notation~\ref{notation_matrices}, we apply an explicit version of Shimura's reciprocity law (see for instance~\cite[Theorem 2.4]{Streng2018}): 
\begin{equation} \label{Shimurareciprocity}
\biggl( \frac{\chi^{70}_{18}}{\Sigma^9_{140}} (Z) \biggr)^\sigma = \biggl( \frac{\chi^{70}_{18}}{\Sigma^9_{140}}\biggr)^{U} (Z')= \frac{\chi^{70}_{18}}{\Sigma^9_{140}}(Z').
\end{equation}
Using Proposition \ref{Theorem:Igusa_single_eq_vanishing_class}, we see that $\chi_{18}(Z)$ vanishes if and only if $\chi_{18}(Z')$ does, and conclude that $(A,E)$ is hyperelliptic if and only if $(A^\frakc,E^\frakc)$ is. To prove Equation~\ref{Conjugate_Shiodas}, we use again Shimura's reciprocity law on the modular functions in the expression of the Shiodas.
\end{proof}

\begin{remark}
The authors believe that a generalization to higher dimensions is still valid. Indeed, to the best of our knowledge, modular forms describing the hyperelliptic locus in higher dimensions are not known. However, the Mumford-Poor Vanishing Criterion still holds, and can be used to obtain similar results. 
\end{remark}  

We now restrict to the case of the modulus $\frakm=(2)$. The following result allows us to compute the Galois conjugates of the Rosenhain invariants. 

\begin{theorem}\label{Theorem:Conjugate_Rosenhains}
  We use the notation~\ref{notation_matrices}. Moreover, we assume that $(A,E)$ is isomorphic to the Jacobian of a marked genus 3 hyperelliptic curve and we consider $\eta$ the azygetic system associated to $Z$. After normalization, let $\lambda_l$, $l=\overline{1,5}$ be the Rosenhain invariants computed with the Takase formula in Theorem~\ref{Th:Takase}. Then for any lift $\tilde{U} = \left( \begin{matrix}
    \tilde{A} & \tilde{B} \\
    \tilde{C} & \tilde{D}
    \end{matrix} \right)\in Sp_{6}(\mathbb{Z})$ of the matrix $U$ with $\delta_0 = \left( \begin{matrix}
   	(\tilde{C}^T\tilde{D})_0\\ 
	(\tilde{A}^T\tilde{B})_0
\end{matrix} \right)$, we have that
\begin{equation}\label{conjugate_Rosenhain}
\lambda_l^{\sigma}= \zeta_4(\tilde{U}, \eta)\cdot (-1)^{\langle (\eta_{l+2})_1,(\eta_1)_2+(\eta_2)_2\rangle} \cdot\lambda'_l, 
\end{equation}
where
\begin{equation*}
\begin{split}
\zeta_4 \left(\tilde{U}, \eta \right) & = \exp \biggl( \pi i \biggl( 
 \phi \bigl( \tilde{U}, \tilde{U}^T \bigl(\eta_{\calU\circ \bigl(\calV\cup\{2, l\}\bigr)} - 1/2 \delta_0\bigr) \bigr) 
+ \phi \bigl( \tilde{U}, \tilde{U}^T\bigl( \eta_{\calU\circ \bigl(\calW\cup\{2, l\}\bigr)} - 1/2 \delta_0\bigr)\bigr) - \\
& \phi \bigl( \tilde{U}, \tilde{U}^T\bigl( \eta_{\calU\circ \bigl(\calV\cup\{1,2\}\bigr)} -1/2\delta_0 \bigr)\bigr) - \phi \bigl( \tilde{U}, \tilde{U}^T\bigl( \eta_{\calU\circ \bigl(\calW\cup\{1,2\}\bigr)} - 1/2 \delta_0 \bigr) \bigr) \biggr) \biggr)^2,
\end{split}
\end{equation*}
and 
\begin{eqnarray*}
\lambda'_l = \left(\frac{\vartheta[\tilde{U}^t\left(\eta_{\calU\circ \left(\calV\cup\{2, l+2\}\right)} - \frac{1}{2} \delta_0\right)]\cdot \vartheta[\tilde{U}^t\left(\eta_{\calU\circ \left(\calW\cup\{2, l+2\}\right)} - \frac{1}{2} \delta_0\right)]}{\vartheta[\tilde{U}^t\left(\eta_{\calU\circ \left(\calV\cup\{1,2\}\right)} - \frac{1}{2} \delta_0\right)]\cdot \vartheta[\tilde{U}^t\left(\eta_{\calU\circ \left(\calW\cup\{1,2\}\right)} - \frac{1}{2} \delta_0\right)]} (Z') \right)^{2}.
\end{eqnarray*}
\end{theorem}
\begin{proof}
Using Theorem \ref{Th:Takase} when the two first coefficients are 0 and 1, the coefficients $\lambda_l$ with $l=  1,\ldots 5 $ can be computed as
\begin{eqnarray*}
\lambda_l = (-1)^{\langle (\eta_{l+2})_1,(\eta_1)_2+(\eta_2)_2\rangle} \left(\frac{\vartheta[\calU\circ (\calV\cup\{2, l+2\})]\cdot \vartheta[\calU\circ (\calW\cup \{2, l+2\})]}{\vartheta[\calU\circ (\calV\cup \{1, 2\})]\cdot \vartheta[\calU\circ (\calW\cup \{1, 2\})]}\right)^2 (Z).
\end{eqnarray*}
For the sake of simplicity let 
\begin{eqnarray*} 
\label{rename_symmetric_diff_elements}
c_1 = \eta_{\calU\circ \left(\calV\cup\{2, l+2\}\right)}, c_2 = \eta_{\calU\circ \left(\calW\cup\{2, l+2\}\right)}, c_3 = \eta_{\calU\circ \left(\calV\cup\{1, 2\}\right)}~\text{and}~c_4 = \eta_{\calU\circ \left(\calW\cup\{1, 2\}\right)}.
\end{eqnarray*} 
By using Shimura's reciprocity law~\cite[Theorem 2.4]{Streng2018}, we have that 
\begin{eqnarray}\label{Def:GalAction_on_Rosenhains}
\begin{split}
\lambda_l^{\sigma} &= \left( (-1)^{\langle (\eta_{l+2})_1,(\eta_1)_2+(\eta_2)_2\rangle} \left(\frac{\vartheta[c_1]\cdot \vartheta[c_2]}{\vartheta[c_3]\cdot \vartheta[c_4]}\right)^2 (Z) \right)^\sigma \\
&=(-1)^{\langle (\eta_{l+2})_1,(\eta_1)_2+(\eta_2)_2\rangle} \left( \left(\frac{\vartheta[c_1]\cdot \vartheta[c_2]}{\vartheta[c_3]\cdot \vartheta[c_4]}\right)^{2} \right)^{U} \left(Z'\right).
\end{split}
\end{eqnarray}

We denote by $ c'_i = \tilde{U}^T\left(c_i - \frac{1}{2} \delta_0\right)$. By applying the theta transformation formula, we get that
\begin{equation*}
\vartheta\left[c_i\right]^U \left(Z'\right)= \vartheta\left[\tilde{U}.c'_i\right] \left(\tilde{U}.Z'\right) = \zeta\left(\tilde{U}\right)\exp{\left(\pi i\phi(\tilde{U}, c_i')\right)} \sqrt{\det(\tilde{C}Z' + \tilde{D})}~ \vartheta\left[c'_i\right]\left(Z'\right).
\end{equation*} 
Hence Equation~\eqref{Def:GalAction_on_Rosenhains} becomes 
\begin{eqnarray*}
\lambda_l^{\sigma} & =  (-1)^{\langle (\eta_{l+2})_1,(\eta_1)_2+(\eta_2)_2\rangle}
 \exp\left(\pi i (\phi(\tilde{U}, c'_1) + \phi(\tilde{U}, c'_2) - \phi(\tilde{U}, c'_3) - \phi(\tilde{U}, c'_4))\right)^2   \left(\frac{\vartheta[c'_1]\cdot \vartheta[c'_2]}{\vartheta[c_3']\cdot \vartheta[c_4']}\right)^{2} 
\left(Z'\right)
\end{eqnarray*} 
where one can easily see that $\zeta_4(\tilde{U}, \eta)=\exp(\pi i (\phi(\tilde{U}, c'_1) + \phi(\tilde{U}, c'_2) - \phi(\tilde{U}, c'_3) - \phi(\tilde{U}, c'_4)))^2 $ is a fourth root of unity.
\end{proof}
\vspace*{0.5 cm}

\noindent
We will now give a geometric interpretation to our results. Recall that the Rosenhain coefficients are invariants for the space $\mathcal{M}_3^{hyp}[2]$. We will show that the Galois conjugates that we compute via Shimura's reciprocity law are in fact the Rosenhain invariants of another point in this moduli space. 

\begin{proposition}\label{Zprime_eta}
  Let $Z\in\calH_3$ be a period matrix of a p.p.a.v. $(A, E)$ with CM by $\calO_K$ such that $A$ is isomorphic to the Jacobian of a hyperelliptic curve, and let $\eta$ be an azygetic system associated to $Z$ after fixing a marking on the curve. Let $[\frakc]\in I_{(2)}(K^r)/H_{(2)}(K^r)$ and let $Z'$ be the period matrix for $(A^\frakc,E^\frakc)$. Assume that $\tilde{U}\in \Gamma_6(1,2)$ . Then $\eta'=\tilde{U}^T\eta$ is an azygetic system associated to the period matrix $Z'$. 
\end{proposition}
\begin{proof}
The action described in Definition~\eqref{Group_Action_On_ppav} yields an isogeny between $(A,E)$ and $(A^\frakc,E)$ which is given by:
\begin{eqnarray*}
I: \C^3/\Z^3+Z\Z^3\rightarrow  C^3/\Z^3+(M.Z)\Z^3,    \ x  \rightarrow  x. 
\end{eqnarray*}
Let $\eta_i$ be the azygetic system giving the map $\eta$. We consider the image of points $(\eta_i)_2+Z(\eta_i)_1 \pmod{\Z^3+Z\Z^3}$
via the isogeny $I$. We compute $\eta'_i$ such that
\begin{eqnarray*}
  (\eta_i)_2+(\eta_i)_1Z=(\eta'_i)_2+(\eta'_i)_1 (M.Z) \pmod {\Z^3+ (M.Z)\Z^3}.
\end{eqnarray*}
  By writing 
$M= \left( \begin{matrix}
    A & B \\
    C & D
    \end{matrix} \right)$
and using that $M.Z=(M.Z)^T$, the 2-torsion point corresponding to $\eta'_i$ is equivalent to:
\begin{eqnarray*}
(ZC^t+D^t)(\eta'_i)_2+(ZA^t+B^t)(\eta'_i)_1=(D^t(\eta'_i)_2+B^t(\eta'_i)_1)+Z(C^t(\eta'_i)_2+A^t(\eta'_i)_1),
\end{eqnarray*}
and thus $\eta'=\tilde{U}^T\eta_i$ is the image of 2-torsion points $(\eta_i)_2+(\eta_i)_1Z$ via the isogeny. It is easy to check 
that this is in fact an azygetic system. The first three facts in Equation~\eqref{Def:Azygetic_base} are trivial to check, the fourth equality follows by applying~\cite[Prop. 13.2(b)]{Milne_AV} for the isogeny $I$, which has degree prime to 2.

To show that $\eta'$ is associated to $Z'$, we will use the Vanishing Criterion.  We choose an even theta characteristic $c\in (1/2)\Z^6$ such that $\vartheta[c](Z)\neq 0$ and $\vartheta[c](Z')\neq 0$ and apply once more Shimura's reciprocity law~\cite{Streng2018} on the quotients of the type $\left (\frac{\vartheta[d](Z)}{\vartheta[c](Z)} \right )^2$, with $d$ even. We deduce that the unique even vanishing characteristic at $Z'$ is $\vartheta[\eta_{\mathcal{U}_{\eta'}}](Z)=0$ (since $\eta_{\mathcal{U}_{\eta'}}=\eta_{\mathcal{U}_{\eta}}$).
\end{proof}

\begin{corollary}\label{U_is_identity}
We use the notation~\ref{notation_matrices}. Then there is a point $Z'\in \Gamma(2)\backslash \mathcal{H}_g$ corresponding to the p.p.a.v. $(A^\frakc, E^\frakc)$ such that the associated marked hyperelliptic curve $X'$ has Rosenhain invariants $(\lambda'_i)_{i=1, \ldots, 5}$ given by $\lambda_i' = \lambda_i^{\sigma}$, for all $i=\overline{1,5}$. 
\end{corollary}
\begin{proof}
We define $C' = BM^T\tilde{U}^T = BM'^T$ with $M' = \tilde{U}M\in GSp_{6}(\Q)$ which is still a symplectic basis with respect to $\bigl(\Phi(N_{\Phi^r}(\frakc)^{-1} \fraka), E_{\Phi, N_{K^r/\Q}(\frakc)\xi}\bigr)$. By reducing mod 2,  we get $\overline{M'} = \overline{\tilde{U}M} = U\overline{M} = I_{6}$ with $\overline{M}\in Sp_6(\Z/2\Z)$ the reduction of $M\pmod 2$. In order to apply Theorem \ref{Theorem:Conjugate_Rosenhains}, let $U' = (\overline{M'})^{-1} = I_6$ in $Sp_6(\Z/2\Z)$. We get that 
\begin{eqnarray}\label{rosenhain}
\lambda^\sigma_i= (-1)^{\langle (\eta_i)_1,(\eta_1)_2+(\eta_0)_2\rangle} \left (\frac{\vartheta[c_1]\cdot \vartheta[c_2]}{\vartheta[c_3]\cdot \vartheta[c_4]}\right)^{2} \left(M'.Z\right),
\end{eqnarray}
for $i\in \overline{1,5}$. 
By Proposition~\ref{Zprime_eta}, since $\eta$ is the azygetic system associated to $Z$, then $U'\eta=\eta$ is the azygetic system associated to $M'.Z$. Hence the right-hand side expressions in Equation~\eqref{rosenhain} are the Rosenhain invariants $\lambda_i'$ of a genus 3 hyperelliptic curve. 
\end{proof}

From a computational point view, if we simply aim at computing the Galois conjugates of the Rosenhain invariants and deriving class field equations, one can choose between the approach in Theorem~\ref{Theorem:Conjugate_Rosenhains} or the one in Corollary~\ref{U_is_identity}. Using the formula in Theorem~\ref{Theorem:Conjugate_Rosenhains}, one can pick any period matrix for $(A^{\mathfrak{c}},E^{\frakc})$, whereas if we use the Corollary~\ref{U_is_identity}, we need to carefully construct the period matrix $Z'$ first.

%% file: classpol_ants.tex
\subsection*{The Rosenhain class polynomials.}

We define the Rosenhain and Shioda class polynomials as follows:
\[ H_{K^r,l}(t) = \prod_\sigma (t - \lambda_l^{\sigma}),~~~S_{K^r,l}(t)=\prod_\sigma (t - (\operatorname{Shi}^{\text{\tiny abs}}(X))_l^{\sigma}(X)),
\]
\noindent
for $1\leq l \leq 5$ and $\sigma, \sigma' \in \Gal(CM_{\frakm}(K^r)/K^r)$ with $\frakm=(2)$ for the product in $H_{K^r,l}$ and $\frakm=(1)$ for the product and sum in $S_{K^r,l}$. For $2\leq l \leq 5$, in order to use the Hecke representation as in Gaudry~\textit{et al}~\cite{Gaudry}, we define:
\begin{eqnarray*}
\hat{H}_{K^r,l}(t) = \sum_\sigma  \lambda_l^\sigma\prod_{\sigma'\neq\sigma} (t- \lambda_1^{\sigma'}),~~~ \hat{S}_{K^r,l}(t) = \sum_\sigma  (\operatorname{Shi}^{\text{\tiny abs}}(X))_l^\sigma\prod_{\sigma'\neq\sigma} (t- (\operatorname{Shi}^{\text{\tiny abs}}(X))_1^{\sigma'}).
\end{eqnarray*}
Algorithm~\ref{Compute_Shimura_Reciprocity} in Appendix~\ref{algorithms} gives all the steps of our computation of a list of approximations for the Galois conjugates of the Rosenhain invariants, that we use to get the polynomials $H_{K^r,l}$ and $\hat{H}_{K^r,l}$. The algorithm for the Shioda class polynomials is similar.

%% file: implementation_ants.tex
\section{Benchmarks and results}\label{implementation}

We implemented the algorithms described here using SAGE~\cite{sagemath} and Magma~\cite{magma} by building on an existing implementation \cite{githubBILV}. The computation of primitive CM types for genus 3 in~\cite{githubBILV} is dependent on the group structure of $\Gal(L/\Q)$. Our CM type computation is independent of this group isomorphism, and works for all genus. In this general setting, we also implemented the construction of the reflex field of $K$ and the reflex CM type using Algorithm~\ref{Compute_Reflex}. Since SAGE does not implement ray class groups, we used an interface to Magma to compute the group $Cl_\frakm(K^r)$ and enumerate elements in $N_{\Phi^r}(Cl_\frakm(K^r))$.

\subsection{Practical experiments}
For space reasons, we reproduce here partially an example and give the full computation in Appendix B. Let $K$ be the CM field defined by the polynomial $x^{6}+43x^{4}+451x^{2}+729$. Since the field contains $i$, all p.p.a.v. with CM by $K$ are hyperelliptic. For one of its primitive CM types, we computed the reflex as the field of equation $x^6 + 1012x^4 + 262048x^2 + 3968064$. The subgroup $N_{\Phi^r}(Cl_\frakm(K^r))$, for $\frakm=(1),(2)$, has three elements, which means that each point will have two Galois conjugates and that the polynomials $H_{K^r,i}$ have degree 3. Tables~\ref{resultsRosenhain} and~\ref{resultsShioda} give minimal polynomials for the coefficients of Rosenhain and Shioda class polynomials, respectively. Table ~\ref{resultsShioda} gives only one of the Shioda class polynomials, the full example is contained in the supplementary material.

\begin{table}[h!]
  \caption{Coefficients of Rosenhain class polynomials for the field of equation $x^6 + 1012x^4 + 262048x^2 + 3968064$.}
  \label{resultsRosenhain}
  \resizebox{1\textwidth}{!}{%
 \begin{tabular}{|c|c|c|c|c|}
\hline
pol. & $t^3$ & $t^2$ & $t$ & $1$\\
\hline
\hline
$H_{K^r,1}$ & $x-1$ & $x^3 + 9x^2 - 48x - 421$ & $x^3 - 96x^2 + 2737x - 22357$ & $x^3 + 43x^2 + 355x + 121$\\
\hline
$\hat{H}_{K^r,2}$ &-  & $9x^3 - 238x^2 + 1361x - 2195$ & $9x^3 - 812x^2 - 45328x - 487744$ & $9x^3 - 7549x^2 + 448286x - 5820221$ \\
\hline
$\hat{H}_{K^r,3}$ & -& $x^3 - 9x^2 - 48x - 25$ & $x^3 - 156x^2 + 3532x - 6424$ & $x^3 - 63x^2 - 3641x - 11825$\\
\hline
$\hat{H}_{K^r,4}$ &- & $9x^3 - 238x^2 + 1361x - 2195$ & $9x^3 - 812x^2 - 45328x - 487744$ & $9x^3 - 7549x^2 + 448286x - 5820221$\\
\hline
$\hat{H}_{K^r,5}$ & - & $x-6$ & $x^3 + 36x^2 - 768x - 26944$ & $x^3 - 192x^2 + 10948x - 178856$\\
\hline
\end{tabular}}
\end{table}

\begin{table}[h!]
  \caption{Coefficients of Shioda class polynomials for the field of equation $x^6 + 1012x^4 + 262048x^2 + 3968064$.}
   \label{resultsShioda}
 \resizebox{1\textwidth}{!}{%
 \begin{tabular}{|c|c|}
\hline
coeff. & minimal pol. \\
\hline
\hline
$t^3$ & $x-1$ \\
\hline
\multirow{4}{*}{$t^2$} & $609125894427130745695834466763740170563639135833980928x^3$\\
 & $+ 767725829025607378425247292652111581405730035262610432x^2$\\
 & $+ 300061222092067234082658423678294482282672624903293536x$ \\
 & $+ 37243744151263324949875407438939777569860345513286901$\\
\hline
\multirow{4}{*}{$t$} & $63402882286988579232480270348050635745503565534880222391610376192x^3$\\
& $- 13725192373693066840488231757093791171761630118575681645149421568x^2$\\
& $+ 786342921318635510916127890581383360136229955111267984417588224x$\\
& $- 13516646075537145153192703242525175243162619024655881644192369$\\
\hline
\multirow{4}{*}{1} & $178186461969600322341142200214605756480742490360904642532008424549206957490176x^3$\\
 & $+ 2500238465575574956316922540016128195983221550816430781122824734688503922688x^2$\\
 & $+ 7942841558044400713140974757114936533108129843365204389947225517213646848x$ \\
 & $+ 6573048087002947388939081561118123324940201519692560411907632812406461$\\
\hline
\end{tabular}}
\end{table}

The coefficients of these polynomials are defined over the real multiplication field of $K^r$, which explains why the degrees of polynomials in the Tables is 3. For most computations on the Rosenhains 500 bits of precision were enough, whereas for the Shiodas we used 5000 bits of precision. Indeed, the modular forms appearing in the expressions of the Shiodas have much larger weight, which results into  much more precision needed when computing with the Shiodas. To compute the Shiodas,  we first computed the Rosenhain coefficients and got an approximation of the equation of the curve, and afterwards computed the Shiodas from this equation. 

All computations were performed on a single core of a Intel Core i7-4790 CPU 3.60GHz and took approximatively 5 minutes at 500 bits of precision and less than 2 hours for 5000 bits. Most time is spent on the theta constants computation, which is performed using the naive implementation in~\cite{githubBILV}. To recognize the values of the coefficients of the polynomials $H_{K^r,i}$ and $S_{K^r,i}$ as algebraic integers, we use the algebraic dependence testing algorithm~\cite{Cohen}, implemented in PARI/GP by the function \texttt{algdep}.
As expected, the polynomials for the Shiodas have larger coefficients, which is due again to the shape of the modular forms in their expression.

\section{Conclusion}
We have introduced class polynomials for invariants of genus 3 hyperelliptic curves and implemented an algorithm for their computation. In upcoming work with XXX\footnote{Name removed for anonymous submission}, we will present examples of sextic CM fields allowing both hyperelliptic and non-hyperelliptic curves with CM by $K$. For such examples, our algorithm for computing a hyperelliptic orbit via Shimura's reciprocity law allows the computation of one equation for the associated class field, whereas the use of invariants of plane quartics on the non-hyperelliptic orbits would provide equivalent equations.

%% file: appendix_ants.tex
\vspace*{-0.3 cm}
\section{Appendix A}\label{algorithms}

\begin{algorithm}[h!]
\caption{Computing the reflex $CM$-pair}
\label{Compute_Reflex}
	\begin{algorithmic}[1]
		\REQUIRE The CM-pair $(K,\Phi)$ and the embeddings $\iota_{K}:K\rightarrow L$, $\pi:L\rightarrow \C$. 
		\ENSURE The reflex CM-pair $(K^r, \Phi^r)$, and the embedding $\iota_{K^r}:K^r\rightarrow L$. 
		\STATE \label{CM_Type_to_Gal_Elem} Compute the inverse CM type $\Phi_L^{-1}$ described in Equation~$\eqref{Lifted_Reflex_Type}$.
		\STATE Compute $H^r < Gal(L/\Q)$ as in Equation~\eqref{Reflex_Field_} and define $K^r:= L^{H^r}$.
		\STATE Set the CM type $\Phi^r$ s.t. $\Phi^{r}_L = \Phi_L^{-1}$. 
		\STATE Compute embedding $\iota:K^r\rightarrow L$ as in Diagram \eqref{Embeddings_Kr_into_L}.
		\STATE \textbf{return} the reflex CM-pair $(K^r, \Phi^r)$ and the embedding $\iota_{K^r}$.
	\end{algorithmic}
\end{algorithm}

\begin{algorithm}[h!] 
\caption{Computing the reflex typenorm}
\label{Type_Norm_Computation}
	\begin{algorithmic}[1]
	  \REQUIRE $\displaystyle (K^r, \Phi^r)$ a primitive CM-pair, the embedding $\displaystyle \iota_{K^r}: K^r \hookrightarrow L$ and a fractional ideal $\displaystyle \mathfrak{a}$ in $K^r$.
          	\ENSURE The image $\displaystyle \mathfrak{b} = N_{\Phi^r}(\mathfrak{a})$ of the reflex typenorm of $\fraka$.
		\STATE Let $\mathfrak{a}' = \iota(\mathfrak{a})\calO_L$ be the lift of $\mathfrak{a}$ to $L$ via $\iota_{K^r}$. 
	\STATE Define $\displaystyle \mathfrak{B}=  N_{\Phi^r_L}(\mathfrak{a}')$.			
        \STATE Define a dictionary $\mathcal{D}$ whose keys are of the form $\frakp$, with $\frakp$ a prime ideal in $\calO_K$, and whose values are lists of couples $(\mathfrak{P},g)$ with $\mathfrak{P}$ above $\frak{p}$ and $g\in \Z$. Set $\mathcal{D}=\{\}$. 
	\STATE \label{prime_decomposition_of_Ideal_L} Compute the prime ideal decomposition   
		$\displaystyle \mathfrak{B} = \prod_{ \mathfrak{P}} \mathfrak{P}^{e}$ and create a set $\displaystyle \mathfrak{M}=\{(\mathfrak{P},e)\}$. 
\FOR{$(\mathfrak{P},e)\in \mathfrak{M}$} 
\STATE \label{Algorithm:Ideal_below} Compute with \cite[Algorithm 2.5.3]{Cohen} the prime ideal $\displaystyle \mathfrak{p}\displaystyle$ in $\mathcal{O}_K$ lying below $\displaystyle \mathfrak{P}$.
\STATE Compute the ramification index $f$ of $\mathfrak{p}$ in $\calO_L$ and $g=e/f$.
\STATE Add $(\mathfrak{P},g)$ to the list for the key $\frakp$ in the dictionary $\mathfrak{D}$.
\ENDFOR 
\STATE The image of the reflex typenorm of $\mathfrak{a}$ is given by
\begin{math}\displaystyle \label{Computing_Ideal_Down}
\frakb = N_{\Phi^r}(\mathfrak{a}) = 
\prod_{(\mathfrak{p}:(\frakP, g))\in \mathfrak{D}} \mathfrak{p}^{g}.
\end{math}
\STATE \textbf{return} $\frakb$.
	\end{algorithmic}	
\end{algorithm}

Let $\frakH = N_{\Phi^r}(Cl_{(2)}(K^r)) \subseteq Cl_{(2)}(K)$. Every $\frakc \in \frakH$ corresponds to an representative in the quotient $I_{\mathfrak{m}}(K^r)/H_{\mathfrak{m}}(K^r)$ via Lemma~\ref{Kernel_Reflex_Type_Norm}. For the sake of simplicity, we denote by $\frakc$ both the element in $I_{\mathfrak{m}}(K^r)/H_{\mathfrak{m}}(K^r)$ and the one in $\frakH$.
\begin{algorithm}[h!]
\caption{Computing the Galois action using Shimura's reciprocity law}
\label{Compute_Shimura_Reciprocity}
	\begin{algorithmic}[1]
	\REQUIRE A $CM$ pair $(K,\Phi)$, where $K$ is a sextic CM field and $\Phi$ is a CM type, and precision $prec$. 
	\ENSURE The Rosenhain class polynomials, if a hyperelliptic curve with CM by $(K,\Phi)$ exists. 
	\STATE Let $\calR_i$, $1\leq l \leq 5$ be an empty list.   
	\STATE Compute the Galois closure $L$ of $K/\Q$.
	\STATE Call Algorithm~\ref{Compute_Reflex} to get the reflex $CM$-pair $(K^r,\Phi^r)$ and the fixed embedding $\iota_{K^r}:K^r\rightarrow L$.
	\STATE Determine the ray class group $Cl_{\frakm}(K^r)$ for the modulus $\frakm = (2)$.
	\STATE \label{Image_TypeNorm} Compute the image of $Cl_\frakm(K^r)$ under the reflex typenorm, and store the result in a list $\frakH$ of elements in $N_{\Phi^r}(Cl_{\frakm}(K^r))$. 
	\STATE \label{choosing_ppav} Choose a p.p.a.v. $\calA$ of dimension $g$ with CM by $\calO_K$ given by the triple $(\Phi,\fraka, \xi)$ and construct period matrix $Z$ with \cite[Algorithm 2]{BILV}. 
	\IF{exactly one of the theta constants $\vartheta[c](Z)$, with $c$ even, is zero} \label{compute_Theta_values}
        \STATE \label{Takase_formula}Compute the Rosenhain invariants $\lambda_i$ with precision $prec$ using Takase's formula~\eqref{Th:Takase}        
	\FORALL{$\frakc\in\frakH$}
	\STATE\label{computeation_conjugate_ppav} Compute p.p.a.v. $\calA^\frakc(\Phi, \fraka, \xi)$ and the corresponding $Z'$.
	
	\STATE Compute $\lambda_l^{\sigma}$  using the formula in Theorem~\ref{Theorem:Conjugate_Rosenhains} and add it to the list $\calR_l$.
	\ENDFOR
	\ENDIF 
        \STATE \textbf{return} $\calR_l, 1\leq l \leq 5$.
        
        \end{algorithmic}
\end{algorithm}

%% file: supplementary_material.tex
\section{Appendix B}
We give here the rest of the computation for the example presented in Section 5 in the paper. The tables below contain the remaining minimal polynomials for the coefficients of Shioda class polynomials for the field of equation $x^6 + 1012x^4 + 262048x^2 + 3968064$. Since the CM field $K$ contains $i$, the odd index Shioda invariants for curves with CM by this field are 0, so we only give the class polynomials for the even index ones.

\begin{table}[h!]
 \resizebox{1\textwidth}{!}{%
 \begin{tabular}{|c|l|}
 \Xhline{2\arrayrulewidth}
 \multirow{2}{*}{coeff.} & \multirow{2}{*}{\hspace*{8.5cm} $S_4$}\\
  & \\
 \Xhline{2\arrayrulewidth}
 $t^3$ & $x-1$ \\
\hline
 \multirow{4}{*}{$t^2$} & $609125894427130745695834466763740170563639135833980928x^3$\\
  & $+ 767725829025607378425247292652111581405730035262610432x^2$\\
  & $+ 300061222092067234082658423678294482282672624903293536x$ \\
 & $+ 37243744151263324949875407438939777569860345513286901$\\
\hline
\multirow{4}{*}{$t$} & $63402882286988579232480270348050635745503565534880222391610376192x^3$\\
 & $- 13725192373693066840488231757093791171761630118575681645149421568x^2$\\
 & $+ 786342921318635510916127890581383360136229955111267984417588224x$\\
& $- 13516646075537145153192703242525175243162619024655881644192369$\\
\hline
\multirow{4}{*}{1} & $178186461969600322341142200214605756480742490360904642532008424549206957490176x^3$\\
& $+ 2500238465575574956316922540016128195983221550816430781122824734688503922688x^2$\\
& $+ 7942841558044400713140974757114936533108129843365204389947225517213646848x$ \\
& $+ 6573048087002947388939081561118123324940201519692560411907632812406461$\\

\Xhline{2\arrayrulewidth}
\multirow{2}{*}{coeff.} & \multirow{2}{*}{\hspace*{8.5cm} $S_6$}\\
  & \\
 \Xhline{2\arrayrulewidth}
  $t^3$ & $x-1$ \\
 \hline
 \multirow{4}{*}{$t^2$} & $36305243238982413890281145728640119544631251519339927191004989892595083547500347392x^3$\\
 & $-4562099625444542640326757150351501343961608027306073167762565630692397127965147136x^2$\\
 & $-38373945786414575731513083189904988026800112808692449136937354443072398826315008x$\\
 & $+ 4414650961793693104140503198846514725940832560953694614278134162355358444958089$\\
 \hline
 \multirow{4}{*}{$t$}& $17918396199353355049058805088857384488383405667618139921548515288945465097396579931049284021849686016x^3$\\
 & $- 127779926553579818761843900212549793508900714146887992364081265379646615684431729282349219465658368x^2$ \\
 & $+ 173870801973972034188665861925383203311853045841476017423892114350599398335858780136987512078336x$\\
 & $-34420059650825478884547414485898422831449454499107719097111434114544516620821487531714046653$\\
 \hline
 \multirow{8}{*}{1} & $2387771057358159367290182151060570674330169223812328868107571288339171666253440282540613475209$\\
 & $61296035059355070744231936x^3$\\
 &$- 252803368000795481224137548975625269232188032444010506173669474186687453303016936277559608353$\\
 & $37040318608580594892800x^2$\\
 & $+ 5398401954573885778797871623679405939013291830944663536859088092482449919401839107837170565278$\\
 & $01421676551340032x$\\
 & $- 3173319027571342197457569660707432734035037164408107365485384034562528494733387662989363676838$\\
 & $10259545343$\\
 \Xhline{2\arrayrulewidth}
\multirow{2}{*}{coeff.} & \multirow{2}{*}{\hspace*{8.5cm} $S_8$}\\
  & \\
\Xhline{2\arrayrulewidth}
 $t^3$ & $x-1$ \\
\hline
\multirow{8}{*}{$t^2$}& $17896610229573747174863113237715437362555551124007560823368715534964490124049106086376218475171$\\
& $803998476828672000x^3$\\
& $- 8632764461526815390006506282494230116166385098236548909260694165812944738921272704072567622170572$\\
& $8388910284800x^2$\\
& $- 16025644434941891793190025881111099657866879313532342984394763561573741656582305902336968805129475$ \\
& $94035269120x$\\
& $+ 32529447226138938988669905527920116558625244669244753609865721673571430611095621276909587434858843$ \\
& $21574283$\\
\hline
\multirow{8}{*}{$t$} & $1282932435619698492314773507177627929574723682784816214540704700755412687484862774112813753070$\\
&$4707515861536346952627346410280321024000000x^3$\\
& $- 8299716559598714066625123496581422312446576620002557854941277709128003194818010807059424517$\\
&$07026350404833519933010891504909025280000x^2$\\
& $+ 11917877361233953970010282512530202208569610409193159992562994892937507196265843185621505728$\\
&$014209932930350406577902992615014400x$\\
& $- 7123291218140354671417349541795074255778101849191965430930690889917518777576393756582739603434$\\
& $903409080046466500094184843$\\
\hline
\multirow{8}{*}{1} & $919680103243938973033922868060563504875468133020309148624843871674616915700907555399703648849339503$\\
& $9639144509259910868729803635874949685516535079108608000000000x^3$\\
& $- 73559035342600500295040776192553962202449203791679388033633296324431166276281136614284845993667$\\
& $7866564005446966040649993887769623302774737002823680000000x^2$\\
& $+ 124859453246417802749874443519286834888515955015727839481478027559413452197186695139754181396863$\\
& $30771129612045565781011078163726078946693873664000x$\\
& $- 93156216301296264623691984297496230976086049395125259735730164341146872241062442916382140292172050$\\
& $4373610916963617612989963293561961647$\\
\Xhline{2\arrayrulewidth}
\multirow{2}{*}{coeff.} & \multirow{2}{*}{\hspace*{8.5cm} $S_{10}$}\\
  & \\
\Xhline{2\arrayrulewidth}
$t^3$ & $x-1$\\
\hline
\multirow{8}{*}{$t^2$} & $355559112646711805441533173187873498621252899864733702350096028638947$\\
& $086434181869745058736889518310522032532646375388879517143627083022336000x^3$\\
& $- 1746673407594299251126586578459915177303446574135079461748793747543688357497$\\
& $51251268047890484907368544601828099865254929376287650532556800x^2$\\
& $- 221926037560681297097463583452946777526823575009037624736284810874206769091545575$\\
& $5758535100520982872331063889546506421674505437288222720x$\\
& $- 12712015828149450830947991623508633378439516962838487229521551449205444676181757419$\\
& $6317114842504230031245936394117370332526149643871$\\
\hline
\multirow{8}{*}{$t$} & $120857237869098796749243849386021434183238111564262812773854488641003266260560109723706082$\\
& $3605269494853912985274275370629735011102984094101536010019521818373425987584000000x^3$\\
& $- 44031291431601012366252978980559805207664034353140198557448777193058697175779212464$\\
& $31509460802917257291378017545643516518179640662145686947311793204195606596157440000x^2$\\
& $+ 39633030527996127828036163516427799270794225874823970274200867418047856664450955053469$\\
& $30905281877052732103313929610944694425132139827112711532749595268140236800x$\\
& $+ 7573459660397995154627991279690982518993463138446961229699989343988200478541388188029$\\
& $8693671007527321130879566194815441715447240349610824132088378321443$\\
\hline
\multirow{8}{*}{1} & $123240874097869995256375628744147413507353413780686466185600538844710585952172273671281859545131359$ \\
&$62104749046597816477320750767810020106355459964636651413138949876841656437080639980977612914688000000000x^3$\\
& $+ 631734108594739263421048849111225980235838610319305095154878283895084698019513801345125252253266$\\
& $8191388231529613779472209762512685540915211230224111283342447639256324862020608604973826048000000x^2$\\
& $+ 40831140498524759788167285099786677735212333529675141992372211716893625351825268656408230063653920$\\
& $8808658845598223352036086838810875696551081698050828145247131187273402074829684736000x$\\
& $+ 289760817269051161710247539111124531055540151755308643194080087839211242064815461147954$\\
& $938189924314930566902032346905632969933471189766079230515290335739274989177970677409$\\
\hline
\end{tabular}}
\end{table}